\documentclass[reqno]{amsart}
\usepackage{amsmath, amssymb, latexsym, amsthm, amsfonts, mathrsfs,
  amscd,color, bbm}
\usepackage{scalerel,stackengine}
\stackMath
\newcommand\reallywidehat[1]{%
\savestack{\tmpbox}{\stretchto{%
  \scaleto{%
    \scalerel*[\widthof{\ensuremath{#1}}]{\kern-.6pt\bigwedge\kern-.6pt}%
    {\rule[-\textheight/2]{1ex}{\textheight}}
  }{\textheight}%
}{0.5ex}}%
\stackon[1pt]{#1}{\tmpbox}%
}
\renewcommand{\eqref}[1]{(\ref{#1})}   

\numberwithin{equation}{section}

\theoremstyle{plain}
\newtheorem{theorem}{Theorem}[section]
\newtheorem{lemma}[theorem]{Lemma}
\newtheorem{corollary}[theorem]{Corollary}
\newtheorem{proposition}[theorem]{Proposition}

\theoremstyle{definition}
\newtheorem{definition}[theorem]{Definition}
\newtheorem{remark}[theorem]{Remark}

\theoremstyle{definition}

\newcommand{\R}{{\mathbbm R}}
\newcommand{\Z}{{\mathbbm Z}}
\newcommand{\C}{{\mathbbm C}}

\newcommand{\F}{{\mathbbm F}}
\newcommand{\A}{{\mathbbm A}}

\newcommand{\ve}{\varepsilon}
\newcommand{\la}{{\langle}}
\newcommand{\ra}{{\rangle}}

\newcommand{\modp}{(\textnormal{mod }p)}
\newcommand{\modq}{(\textnormal{mod }q)}
\newcommand{\ep}{\varepsilon}

\newcommand{\tr}{\textnormal{Tr}}
\newcommand{\deter}{\textnormal{Det}}

\begin{document}

\title[Non-abelian Polya-Vinogradov inequality]{Singular Gauss sums,
 Polya-Vinogradov inequality for $GL(2)$ and growth of primitive elements}
\author{Satadal Ganguly}

\address{Indian Statistical Institute, Kolkatta 700108,  INDIA.}  
\email{sgisical@gmail.com}

\author{C.~S.~Rajan}

\address{Tata Institute of Fundamental  Research, Homi Bhabha Road,
Bombay - 400 005, INDIA.}  \email{rajan@math.tifr.res.in}


\begin{abstract} We establish an analogue of the classical
Polya-Vinogradov inequality for $GL(2, \F_p)$, where $p$ is a
prime. In the process, we compute the `singular' Gauss sums for $GL(2,
\F_p)$.  As an application, we show that the collection of elements in
$GL(2,\Z)$  whose reduction modulo $p$ are of maximal order in $GL(2,
\F_p)$ and whose matrix entries are bounded by $x$  has the expected
size as soon as  $x\gg p^{1/2+\ep}$ for any  $\ep>0$.
\end{abstract}

\maketitle

\section{Introduction}
Let $\chi$ be a non-principal
Dirichlet character of modulus $q$. The well-known
Polya-Vinogradov estimate for character sums is given by
(see \cite[Ch. 23]{Dav})
\begin{equation}\label{pv}
\sum_{x \leq n< x+y}\chi(n) =O( \sqrt{q}~\mbox{log} ~q),
\end{equation} 
where $x$ and $y>0$ are any integers. Here, by the notation
$ f(y) =O(g(y))$ or $f(y) \ll g(y)$, $f$ being any function and  
$g$ being a positive function defined on a domain
$Y$,    
we mean that there is a constant $c>0$ such the bound
$\left|f(y)\right| \leq c g(y)$ holds for all $y\in Y$. 
The trivial
bound for such a character sum is $y$ and one can easily obtain the
bound  $q$.  Thus the Polya-Vinogradov bound indicates cancellations 
in a sum of the character values along an
interval as soon as the length of the interval becomes somewhat larger than
$\sqrt{q}\log q$.

It is natural to consider what the analogue of the Polya-Vinogradov
estimate should be for groups more general than $(\Z/q\Z)^*$. To
start with, one may consider the following broad question: \\
Do cancellations occur in a sum of 
the type
\begin{equation}\label{pvgen}
 \sum_{H(A)\leq x} \chi_{\rho}(A),
\end{equation} 
where $\rho$ is a non-trivial complex representation of
$G(\Z/q\Z)$, $G$ being a linear algebraic group defined over $\Z$,
$\chi_{\rho}=\tr\circ \rho$ its character (where $\tr$ denotes the
\textit{trace} map), and $H: G(\Z) \to \R$ is some suitable height function
that measures the
`size' of $A$? An affirmative answer would amount to obtaining
non-trivial bound for this sum in terms of $q$ that is uniform over
$x$. 

The proofs of the classical Polya-Vinogradov bound 
guides us to the groups to which we should attempt to generalize it. 
All the proofs
utilize harmonic analysis on the abelian group  $\Z/q\Z$ in one way or
the other. For example, one can  expand the Dirichlet character in a
finite Fourier series in terms of the additive characters where the
Fourier coefficients are essentially the classical Gauss sums. Then
one needs to estimate a finite geometric series and use the classical
bound $O(\sqrt{q})$ for Gauss sums to obtain \eqref{pv}. 

The analogy of Gauss sums with $L$-functions and the use of abelian
harmonic analysis in the method of Tate-Godement-Jacquet for proving
analytic properties of the standard $L$-functions attached to cusp
forms on $GL(n)$  suggests that a natural generalization should be to
the group $GL(n, \Z/q\Z)$.  We restrict to the case  $q=p$, an odd
prime, for simplicity.  Apart from the group $GL(n, \Z/p\Z)$ 
being a natural generalization of the
group $GL(1,\Z/p\Z)\simeq (\Z/p\Z)^*$, the key point is that, similar
to  the classical case, we can utilize abelian harmonic analysis on
the additive group $M(n,\Z/p\Z)$, the group of $n \times n$ matrices 
over $\F_p$, to study the sum \eqref{pvgen}. Since the group $M(n,\Z/p\Z)$ is self-dual, one new aspect that arises is the evaluation of  singular Gauss sums, attached to singular matrices in $M(n,\Z/p\Z)$.

\subsection{A GL(2) Polya-Vinogradov bound}
For a matrix $A \in M(n, \Z)$, we denote by
$\bar{A}$ its reduction modulo $p$; i.e., the image of $A$ under the
reduction map $M(n, \Z)\to M(n, \F_p)$. We extend $\chi_{\rho}$ to a
function on $M(n, \F_p)$ by defining it to be zero on matrices whose
determinant vanish modulo $p$ and  consider $\chi_{\rho}$ as a
function on $M(n, \Z)$ by the reduction modulo $p$ map.
We now make a definition which is 
a natural choice for the height function:
\begin{definition}
 For $A \in M(n, \Z)$, we define
$$h(A):= Max\{|a_{ij}|\},$$  $a_{ij}$ being the $(i, j)$-th entry of
$A$.
\end{definition}

Our first main theorem is the following $GL(2)$-analogue of
 the classical Polya-Vinogradov inequality.

\begin{theorem}\label{main}
Let $\rho$ be a non-trivial  irreducible complex representation 
of the group $GL(2, \F_p)$. Let $d(\rho)$ be the dimension of $\rho$.
 Then, for any $x\geq 1$, we have the estimate
 \begin{equation}\label{mainpv}
 \sum_{A \in M(n, \mathbbm{Z}): h(A)\leq x}\chi_{\rho}(A) \ll d(\rho)p^{2}(\log p)^{4},
 \end{equation}
where the implied constant is absolute and one can take it to be $16$ if $p\geq 11$.
 \end{theorem}
 
 We now state a more general version of the theorem from
 which Theorem \ref{main} follows easily. First we define the
 notion of a matrix interval.
 \begin{definition} \label{matint}
 By a \textit{matrix interval} over the integers we
shall mean a set $\mathbf{I}$ of the form  $\mathbf{I}=\prod_{1\leq
i,j\leq n}I_{ij}$, where each $I_{ij}$ is an interval in $\Z$; i.e.,
$I$ is the set of $n \times n$ integer matrices  $((a_{ij}))$ such
that for every fixed  pair $(i,j)$, the $(i, j)$-th entry $a_{ij}$
varies over the  component interval $I_{ij}$ in $\Z$.
\end{definition}

 A simple example of  a matrix interval to keep in mind is to take some fixed 
matrix $A_0$ and define $\mathbf{I}=\mathbf{I}(A_0, x)$ to be the collection of
matrices $A$ such that $h(A-A_0)\leq x$. 

With the above definition, our theorem is:
 \begin{theorem}\label{main2} Suppose $\mathbf{I}=\prod_{1\leq i,j\leq
2}I_{ij}$ is a matrix interval over the integers such that the length
of each component interval satisfies the bound $|I_{ij}|\leq cp$,
where $c>0$ is a constant. Then, under the same assumptions on a
representation $\rho$ as above, we have the bound
\begin{equation}\label{secondmain}
 \sum_{A \in \mathbf{I}} \chi_{\rho}(A) \ll d(\rho)p^{2}(\log p)^{4},
 \end{equation}
where the implied constant is absolute and can be taken to be 
${\left(\frac{c+3}{2}\right)}^4$ if $p\geq 11$.
\end{theorem} 

\noindent
\textbf{Remarks}\\
\noindent
1.  Recall (see Remark \ref{charsumzero}) that if $\chi$ is
a non-trivial character of a finite group $G$ then
\[
\sum_{g\in G}\chi (g)=0.
\]
It follows, therefore, that if $\mathbf{I}=\prod_{1\leq i, j\leq
  2}I_{ij}$ is a matrix interval having component intervals $I_{ij}$ of the type
$I_{ij}=[0, r(p-1)],$
where $r\geq 1$ is a fixed integer, then 
\[
\sum_{A \in \mathbf{I}} \chi_{\rho} (A)=0.
\]
Thus, in the situation of Theorem \ref{main},
we may 
assume that $x <p$ and  apply Theorem \ref{main2} with $c=1$  
to obtain Theorem \ref{main}. 

\noindent
2. 
The trivial estimate for the sums
in Equations \eqref{mainpv} and \eqref{secondmain} is $d(\rho)p^4$, 
which shows that we obtain a `saving'
of $p^2$ compared to the trivial estimate. 

\noindent
3. The dimension $d(\rho)$ can be at most $p+1$ (see \S 2.2) and thus the character sums in the above two theorems are of
size $O(p^3(\log p)^4)$.

\subsection{Non-abelian Gauss sums}
The analogue of the Gauss sums for  $GL(n,\F_p)$ was introduced 
by Lamprecht (\cite{L}). Let $\rho$ be an irreducible, complex
representation of the group $GL(n,\F_p)$ and let $\chi_{\rho}$ be its
character. By $e(z)$ we shall denote $e^{2\pi i z}$ for a complex number $z$ and by $e_p(z)$ we shall denote
$e(z/p)$  throughout. Then, for integers $x$, the map $x \mapsto e_p(x)$ defines 
 an additive character (denoted again by $e_p$) 
on the finite field $\F_p$
identified with $\Z/p\Z$. The bilinear pairing $(A, X) \mapsto e_p\left(
  \tr(AX)\right)$ on  $M(n, \F_p)$, yields an identification of 
 $M(n, \F_p)$ with its dual group of characters. 
Following Lamprecht (\cite{L}), define the (matrix valued) Gauss sum 
attached to $\rho$ and $A$ as:
\begin{equation}\label{gausssum}
 G(\rho, A)=\sum_{X \in G} \rho(X)e_p( \tr(AX)).
\end{equation}
It is easy to verify that for $A \in GL(n, \F_p)$,
\begin{equation}\label{equivariance}
  G(\rho, A) = \rho(A)^{-1} G(\rho, I_d),
\end{equation}
 where $d =d(\rho).$ 
By  Schur's lemma, it follows that $G(\rho, I_d)$
is a scalar matrix, 
\begin{equation}\label{schur}
 G(\rho, I_d)=g(\rho)I_d,
\end{equation}
for some constant $g(\rho)$. 

The characters of the irreducible complex representations of $GL(n, \F_p)$ were
obtained explicitly by Green \cite{G} in terms of the `dual data'
consisting of the conjugacy classes of elements in  $GL(n, \F_p)$. 
Using Green's work, Kondo \cite{Ko} obtained the following estimate 
for  the size of the above  Gauss sums:
\begin{theorem}[Kondo] Let $\rho$ be an irreducible, complex 
  representation of $GL(n, \F_p)$. Then, 
\begin{equation}\label{kondo-exact}
 |g(\rho)|=p^{(n^2-k(\rho))/2},
\end{equation}
where $k(\rho)$ is the generalized  multiplicity of the eigenvalue $1$ in
the conjugacy class attached to $\rho$ by the Green correspondence. 
\end{theorem}
\begin{remark} More precisely, Kondo proves that up to a power of $p$, the non-abelian Gauss sum $g(\rho)$ is actually an `abelian' Gauss sum, attached to a
character of a maximal torus $T$ of $GL(n)$. Kondo's result was also proved by Braverman and Kazhdan (\cite[Theorem 1.3]{BK}), using the construction of irreducible representations of $GL(n,\F_p)$ by Deligne and Lusztig (\cite{DL}) and the theory of character sheaves due to Lusztig.  We recall this result now. 

Let $T$ be a maximal torus of $GL(n)$ over $\F_p$, and $\theta: T(\F_p)\to \bar{Q}_{\ell}^*$ be a character. Associated to this data, Deligne and Lusztig construct a virtual representation $R_T(\theta)$ of $GL(n,\F_p)$, and show that every irreducible representation $\rho$ of $GL(n, \F_p)$ is an irreducible constituent of some  $R_T(\theta)$. Consider the abelian Gauss sum
\[ g(\theta)= \sum_{X \in T(\F_p)} \theta(X)e_p( \tr(X)).\]
Braverman and Kazhdan have shown (\cite[Theorem 1.3]{BK}) that
\[ g(\rho)=p^{(n^2-n)/2}g(\theta).\]

\end{remark}
\begin{remark}
The foregoing result allows us to specify $k(\rho)$. With the notation of Section 2.2,
\[k(\rho)=\begin{cases} 2 \quad \mbox{if $\rho\simeq St$},\\
1 \quad \mbox{ if $\rho\simeq I_{\chi, 1}$ and $\chi$ is non-trivial},\\
0 \quad \mbox{otherwise}
\end{cases}
\]
\end{remark}

\subsection{Singular non-abelian Gauss sums}
Equations (\ref{equivariance}) and (\ref{schur}), gives an
estimate for the trace of $G(\rho, A)$, provided $A$ is a {\em non-singular}
matrix: 
\begin{equation}\label{kondo}
 |\tr(G(\rho, A))|\leq d(\rho)p^{n^2/2}, 
\end{equation}
where $d(\rho)$ is the dimension of $\rho$. 
If $\rho$ is not abelian,
$d(\rho)$ is either $p-1, p ~\mbox{or}
~p+1$. 
from the classification of
irreducible representations of $GL(2,\F_p)$ (see \S 2.2). 
Thus $d(\rho)\leq p+1$ for any representation $\rho$ and $d(\rho)$ is
 of order $p$ unless $\rho$ is abelian.

However, for the purpose of establishing an analogue of the
Polya-Vinogradov inequality, we need to estimate Gauss
sums attached to all (additive) characters  $M(n, \F_p)$; in paricular, we need
to estimate the {\em singular} Gauss sums, 
by which we mean the trace of $G(\rho,
A)$ where $A$ is a singular matrix in   $M(n, \F_p)$. 

It is easy to see that the trace of $G(\rho, A)$
depends only on the conjugacy class of $A$. Let 
\[ A_a =\begin{pmatrix} a & 0\\ 0 &
  0\end{pmatrix},  ~~a\neq 0 \quad \mbox{and}\quad  
N= \begin{pmatrix} 0 & 1\\ 0 &
  0\end{pmatrix}.\]  
A non-zero singular matrix in 
  $M(2,\F_p)$ is conjugate to either $A_a$ or $N$.  One of our main
results is Theorem \ref{theorem:sgs} 
below which gives the explicit
 values of the singular Gauss sums when
 $n=2$. We restrict to the case $n=2$ for simplicity 
and this case is already quite involved. 
We expect that a similar result for general $n$ should hold. 
\begin{theorem} \label{theorem:sgs}
 Let $\rho$ be a complex irreducible representation of $G=GL(2,\F_p)$, and let
 $A$ be a non-zero singular matrix in 
$M(2,\F_p)$.
Then the following statements hold:
\begin{enumerate}
\item \label{vanishingpart}
Suppose $\rho$ is not isomorphic to either the trivial representation $1_G$, or the
  Steinberg representation $St$ or the principal series representation $I_{\chi, 1}$ with 
$\chi$ a non-trivial character of
$\F_p^*$. Then,
\[\tr ( G(\rho, A))=0.\]

\item  For the trivial representation $1_G$, 
\[ G(1_{G},A)=-p(p-1). \]

\item If $\rho\simeq I_{\chi, 1}$ with $\chi$ a non-trivial character of
$\F_p^*$, then 

\begin{align*}
\tr( G(I_{\chi, 1},
  A_a))&=p(p-1)\overline{\chi(a)}G(\chi)\\
\tr( G(I_{\chi, 1},
  N))&=p(p-1)G(\chi),
\end{align*}
where $$G(\chi)=\sum_{a\in \F_p^*}\chi(a)e_p(ax)$$ is the usual classical
Gauss sum.

\item For the Steinberg representation $St$, 

\begin{align*}
\tr( G(St,A_a))&=-p(p-1).\\
\tr( G(St,N))&=p^2(p-1).
\end{align*}
\end{enumerate}
\end{theorem}
\begin{remark}\label{charsumzero}  It follows from the 
orthogonality of characters, that
$\tr( G(\rho, 0))$ vanishes when $\rho$ is a non-trivial
irreducible representation of   $GL(n, \F_p)$, and
equal to $|GL(n, \F_p)|$ if $\rho=1_G$, the trivial representation.
\end{remark}

As a consequence of the above result and Kondo's estimate for
non-singular Gauss sums given by Eq. (\ref{kondo}),  
the following general theorem is immediate after one applies 
the Gauss estimate for the classical Gauss sum:
$|G(\chi)|=\sqrt{p}$ and recalls the fact that the dimensions 
of $I_{\chi, 1}$ and $St$
are, respectively, $p+1$ and $p$ (see \S 2). 
\begin{theorem}\label{gaussest-general}
Let $p>2$ be a prime and let $\rho$ be a complex irreducible
representation of $GL(2,\F_p)$ and let $A$ be a non-zero matrix in 
$M(2,\F_p)$. Then, 
\begin{equation}\label{eqn:generalest}
 |\tr(G(\rho, A))|\leq d(\rho)p^{2}. 
\end{equation}

\end{theorem}

\subsection{Applications of the $GL(2)$ Polya-Vinogradov inequality} 
We first describe the general plan for applications here. Let $\phi$ be a  $GL(2,\F_p)$ conjugacy-invariant function on
$M(2,\F_p)$. Consider
the sum, 
\begin{equation}\label{sumdef}
S(\phi, x)=\sum_{ h(A)\leq x}\phi(\overline{A}),
\end{equation}
where $\overline{A}$ denotes $A \modp$.
Decomposing $\phi$ as a Fourier series in terms of the
irreducible characters of $G$, we write
\[\phi=\sum_{\rho\in \hat{G}} c_{\phi}(\rho)\chi_{\rho},\]
where $\hat{G}$ is the collection of  complex irreducible representations of
$G$ up to isomorphism and $$ c_{\phi}(\rho)=\frac{1}{|G|}\sum_{g\in
  G}\phi(g)\overline{\chi_{\rho}(g)}$$ is the Fourier coefficient of
$\phi$ with respect to the character $\chi_{\rho}$.  From Theorem
\ref{main}, upon singling out the contribution from the trivial
representation of $G$ as the `main term', we obtain the estimate
\begin{equation}
 \sum_{ h(A)\leq x}\phi(A) =c_{\phi}(1_G) \sum_{h(A)\leq x}
 \chi_1(A)
+O\left( d(\rho)p^{2}(\log
  p)^{4}\left|\sum_{\rho\in \hat{G}}c_{\phi}(\rho)\right|\right), 
\end{equation}
where for  simplicity of notation, we write
$\chi_1$ to denote the trivial character of $G$. By Lemma \ref{countingns}, the contribution of  the trivial
character is,  
\begin{equation}
\sum_{h(A)\leq x} \chi_1(A)=16\gamma_p x^{4}+O(x^{3}),
\end{equation}
where $\gamma_p=1-\frac1p-\frac1{p^2}+\frac1{p^3}.$

Thus we obtain the general formula 
\begin{equation}\label{generalfunctionest}
 \sum_{ h(A)\leq x}\phi(A) =16 c_{\phi}(1_G)\gamma_p x^{4}+O(c_{\phi}(1_G)x^{3})
+O\left( d(\rho)p^{2}(\log
  p)^{4}\left|\sum_{\rho\in \hat{G}}c_{\phi}(\rho)\right|\right).
\end{equation}

\subsubsection{Counting elements 
in a conjugacy class} We now consider the case where $\phi$ is the
characteristic function of a conjugacy class $C$ in $G=GL(2,
\mathbbm{F}_p)$.   
We want to count the number of matrices $A$ in
$M(2, \mathbbm{Z})$ 
 with height bounded by $x$ that reduces modulo $p$ to an element 
lying in $C$.  

Let $\delta_C$ we denote the 
indicator function of the subset $C$ of $G$. By orthogonality of characters, 
\[ \delta_C = \frac{|C|}{|G|}\sum_{\rho\in \hat{G}}
\overline{\chi_{\rho}(c)}\chi_{\rho},
\]
for any  $c\in C$.  
Therefore, proceeding as before, we have,
\begin{align*}
S(\delta_C, x)= \sum_{ h(A)\leq x} \delta_C (A)& =  \frac{|C|}{|G|} 
\sum_{\rho\in \hat{G}}\overline{\chi_{\rho}(c)}\sum_{h(A)\leq x} \chi_{\rho}(A)\\
 &=\frac{|C|}{|G|}\sum_{h(A)\leq x} \chi_1(A)+\frac{|C|}{|G|} 
\sum_{\rho\neq 1_G}\overline{\chi_{\rho}(c)}
 \sum_{h(A)\leq x} \chi_{\rho}(A);
\end{align*}

and we obtain the following general statement:
\begin{proposition}\label{thm:conjugacyclass}
Suppose $C$ is a conjugacy class in $G=GL(2, \mathbbm{F}_p)$ and 
$c \in C$ is any element,  we have the equality
 \begin{equation*}
 S(\delta_C,x) =
  \frac{16|C|\gamma_p}{|G|}x^{4}+O\left(\frac{|C|}{|G|}\left(x^3+p^{2}(\log
      p)^{4}
\sum_{\rho \neq 1_G}d(\rho) |\overline{\chi_{\rho}(c)}|\right)
\right).
 \end{equation*}

\end{proposition}

\begin{remark}
The above result is of limited use as the inner sum $\sum_{\rho \neq 1_G}d(\rho) |\overline{\chi_{\rho}(c)}|$ can be quite large
in general. However, for certain conjugacy classes  this simple approach  already gives a non-trivial result. See the next subsection for an example. 
\end{remark}

\subsubsection{Elliptic elements} An element in $GL(2, \F_p)$ is said
to be \textit{elliptic} if its characteristic polynomial is irreducible over
$\F_p$. We shall call an integer matrix elliptic if its reduction modulo $p$ is elliptic. The problem of finding 
 an elliptic element of the least height  can be considered in  analogy with the classical problem of
finding the least quadratic non-residue for a prime $p$ (see
\cite{Mont}). It follows from \eqref{pv} that for any
$\ep>0$,  there is a positive integer $\tau=O(p^{\frac{1}{2}+\ep})$
that is a quadratic non-residue for the prime $p$ and the matrix
$\bigl(\begin{smallmatrix} 0 &\tau \\ 1 & 0
\end{smallmatrix} \bigr)$ is an elliptic element of height 
$O(p^{\frac{1}{2}+\ep})$. Henceforth, we shall follow the standard custom of using the symbol $\ep$ to denote a positive
real number which will be assumed to be as small as we please and the value of $\ep$ may differ
from one occurrence to the other. 

Now, suppose we want to count the elliptic elements 
of height up to $x$. Let $\Omega_e$ denote the set of elliptic elements in $GL(2, \F_p)$.
Therefore, we need to estimate the size  of $S(\delta_{\Omega_e},x)$. Following the proof of  Prop. \ref{thm:conjugacyclass}, we can easily obtain a result of
the form 
$$S(\delta_{\Omega_e},x)=8\left(1-\frac{2}{p}+\frac{1}{ p^2}\right) x^4+O(x^3+ p^{3+\ep}),$$
which shows that asymptotically half of all matrices reduce to
elliptic elements modulo $p$ as soon as $x\gg p^{3/4+\ep}$. However,
by a direct and simple argument using the classical Polya-Vinogradov bound for characters of $\F_p^*$, we establish the following easy result which shows that it is enough to take $x\gg
p^{1/2+\ep}$:
\begin{proposition}\label{thm:growth-elliptic}  With notation as above, 
\[ S(\delta_{\Omega_e},x)=8\left(1-\frac{2}{p}+\frac{1}{ p^2}\right) x^4 + O(x^3 \sqrt{p}\log p).\]
\end{proposition}
This theorem is used in the problem of estimating the growth of  the number of
primitive elements of height up to $x$ described in the next section.

\begin{remark}
If we use the Burgess bound (see \cite{Bu, Bur2, Bur3}) instead of the Polya-Vonogradov bound, then it is possible to obtain a superior result
but that does not lead to any improvement in the final application towards counting primitive elements. 
\end{remark}

\subsection{Application to counting Primitive elements}
Given a prime  $p$, assumed to be large, a classical problem is to estimate the
 size of  the smallest positive primitive root
$g_p$ (i.e., 
a generator for the cyclic group $\F_p^*$).
This can be reduced to a question of estimation of character sums and
by the celebrated bound of Burgess \cite{Bu} on character sums, 
one can show that (see \cite{Mont})
$$
g_p \ll_{\ve}p^{\frac{1}{4\sqrt{e}}+\ep}.
$$
For $G=GL(2,\F_p)$, we consider the generators of the subgroup
$\F_{p^2}^*$ as   analogue of the primitive roots for $\F_p^*$. 
Such elements are the elliptic semisimple elements (see \S 5) of order 
$p^2 -1$, which is  the maximum possible order in $G$. We shall refer to them as 
{\em primitive elements}.

Let $\Omega_{prim}$ denote the 
set of primitive elements in $G$. One has (see \S 2.1)
\begin{equation}\label{primitive}
\frac{|\Omega_{prim}|}{|G|}=\frac{\phi(p^2-1)}{2(p^2-1)}.
\end{equation}
By the observation that $4$ divides $p^2-1$ and by the lower bound
 $\frac{\phi(n)}{n} \gg (\log \log n)^{-1}$
(see \cite[Thm 15]{RS}), we have the following bounds for the above ratio:
\[
(\log \log p)^{-1}\ll \frac{|\Omega_{prim}|}{|G|}\leq \frac{1}{4}. 
\]
Here $\phi$ denotes the Euler $\phi$-function.
We have used the representation theory of $GL(2, \F_p)$, Theorem \ref{main2}, and the classical Polya-Vinogradov estimate to prove
the following theorem which gives an asymptotic formula  for the number of elements  
in the set $\{A \in M(2, \Z): h(A) \leq
x\}$ that reduce  to  primitive  elements modulo $p$. 

\begin{theorem}\label{lqprthm} For any $\ep>0$, we have
 \begin{equation}\label{strong}
S(\delta_{\Omega_{prim}}, x)= \frac{8\phi(p^2-1)}{(p^2-1)}(1-2/p+1/p^2)x^4+O(x^3 \sqrt{p}\log p)+O(x^2p\log p)+O(p^{2+\ep})
\end{equation}
\end{theorem}

The following is immediate:
\begin{corollary}\label{lqpr}
Given a sufficiently large but fixed prime $p$ and any $x\gg p^{1/2+\ep}$, 
 a positive proportion of the set of matrices of height up to $x$
reduce to primitive elements of $GL(2, \F_p)$. In
particular, there is a matrix of height $O(p^{1/2+\ep})$ that reduces 
to a primitive element of $GL(2, \F_p)$.
\end{corollary}

\begin{remark}
An interesting question is whether one can prove the existence of primitive elements in a one-parameter family of the form $\mathcal{A}=\{B+nI: 1\leq n \leq x\}$, where  $B$ is some suitable fixed matrix and $x>0$ is a parameter that we want to make as small as possible relative to $p$ (for example, $x=p^{1/2+\ve}$ would be a natural choice). In other words, we would like to know whether there is an integer $n$ which is not too large such that the eigenvalues of the matrix $B+nI$ are primitive roots
for $p^2$ (i.e., generators of the cyclic group $\F_{p^2}^{\ast}$). 

Assume that the characteristic polynomial of $B$ is not reducible over $\F_p$ and that  $\theta_1$ and $\theta_2$ are the eigenvalues of $B$. 
Then the eigenvalues of $B+nI$ are  $\theta_1+n$ and $\theta_2+n$, and thus we are led to the following general question:
 
 Suppose $q=p^m$, and $\theta$ is an element of $\F_{q}$ such that $\F_p(\theta)=\F_q$.  Is there some element $a\in \F_p$ such that $\theta+a$ is a primitive root 
for $q$ and if so, how small can we take $a$ to be (identifying the elements of $\F_p$ with integers from $0$ to $p-1$)? 

The study of such questions was
initiated by Davenport \cite{Dav2} and there have many works subsequently, e.g.,  \cite{Dav-Lew} and  \cite{Bur2}, to name a few. In \cite{PS}, Perel'muter and Shparlinski count the number of primitive roots for $q$
 in a set of the form $\{\theta+n:0\leq n \leq X\}$. It follows from their result that there are integers $n=O(p^{1/2+\ep})$ such that $\theta+n$ is a primitive root for $q$. This proves the existence of primitive matrices in one-parameter families of the form $B+nI$ with $n=O(p^{1/2+\ep})$, provided that
the characteristic polynomial of $B$ is irreducible over $\F_p$. 

Now, using Prop. \ref{thm:growth-elliptic} and the work of \cite{PS}, it is possible by a careful analysis to 
give an alternative proof of  Theorem \ref{lqprthm} and we have carried it out in \S \ref{Perel}. The error term we get by this 
method is a little different but there is no substantive change in the strength of the result. 

We emphasize here that the result in \cite{PS} depends
crucially on the Riemann Hypothesis for curves over finite field proved by Weil, whereas the first proof of Theorem \ref{lqprthm} we have given
in \S 6 using representation theory requires
no tool form Algebraic Geometry. A curious feature of the representation-theoretic proof is that the main term results not from the contribution of  the trivial representation alone and both the trivial representation and the Steinberg representation have to be considered together to obtain the main term. 
 \end{remark}
 
\begin{remark}
In view of Theorem \ref{main2}, one can replace the sum 
$S(\delta_{\Omega_{prim}}, x)), x)$ by the sum
$S(\delta_{\Omega_{prim}}, x)),A_0, x):=\sum_{h(A-A_0)\leq
  x}\delta_{\Omega_{prim}}(A)$ 
and arrive at a similar
estimate, where $A_0$ is some chosen base matrix. 
\end{remark}

\begin{remark} When $x$ is small, namely if $x<p$, the cofficient of the 
main term in Prop. \ref{thm:conjugacyclass} (resp. Theorem 
\ref{thm:growth-elliptic}, Theorem \ref{lqprthm}) can be taken to be 
$16|C|/|G|$ (resp. $8$, $8\phi(p^2-1)/(p^2-1)$).  
The correction factor $\gamma_p$ (resp. $(1-2/p+1/p^2)$, $(1-2/p+1/p^2)$) 
arises for larger $x$, due to 
the contribution from matrices that reduce to singular matrices
modulo $p$.
\end{remark} 

\subsection{Some general remarks}

\noindent
1. It will be interesting to extend our results to
  $GL(2, \Z/q\Z)$ for an arbitrary positive integer $q$. If $q$ is
  square-free, this group is a product of groups of the form
  $GL(2,\F_p)$ for primes $p$ dividing $q$, and the  
 irreducible representations of  $GL(2,\Z/q\Z)$ is 
 a tensor product of the irreducible representations of $GL(2,\F_p)$

\noindent 2. In the case of a Dirichlet character $\chi \modq$,  the
Polya-Vinogradov bound indicates cancellations as soon as the length
$X$ of the sum $\sum_{n \leq X} \chi(n) $ is somewhat larger than
$\sqrt{q}\log q$. However, cancellations do take place in sums of much
shorter length and cancellations in such shorter sums correspond to
strong bounds on the Dirichlet $L$-function.  Indeed, showing
cancellations in a sum of length $O(q^{1/2-\delta})$ for any
$\delta>0$ amounts to proving a subconvex estimate for $L(s, \chi)$
(see \cite[Chap. 5]{IK}) and the greater the value of $\delta$ we can
take, the  stronger is the bound on the $L$-function. In particular,
Lindel\"{o}f Hypeothesis on $L(s, \chi)$ corresponds to cancellations
in extremely short sums of length $O(q^{\ep})$ for any $\ep>0$. It
will be very interesting to develop of a theory of $L$-function
attached to a representations of $GL(n, {\F}_p)$ in order to study the
sums we are considering. It is not clear within what height we should
expect to find cancellations in the sums over matrices and, in
particular, whether the analogue of Linde\"{o}f hypotehsis should
hold. Any theory, even a conjectural one, for making a deeper analysis
of these sums will be welcome.

\noindent 3.  There are several natural choices for  a  height
functions other than the one considered here; e.g., the operator norm
or the $L^2$-norm of a matrix. It would be interesting to investigate
whether one could obtain similar results with other height functions.

\subsection{Main ideas behind the proofs and the structure of the
paper} The proof of Theorem \ref{main2} follows the usual approach for
proving the classical Polya-Vinogradov inequality. The periodicity of
$\chi_{\rho}$ allows one to consider the sum 
\[S(\chi_{\rho}, \mathbf{I}):=\sum_{A\in \mathbf{I}}\chi_{\rho}(A),\] 
as an inner product $<\chi_{\rho}, \delta_{\bar{\mathbf{I}}}>$ 
on the group $M(n,\Z/p\Z)$, where
  $\bar{\mathbf{I}}$ is the image of $\mathbf{I}$ under the natural projection map from $M(n, \Z)$ to
$M(n,\Z/p\Z)$. Applying the isometry of the Fourier
transform on  $M(n,\Z/p\Z)$, the problem reduces to that of estimating
two kinds of sums: sums of additive characters that lead to finite
geometric sums, and the matrix Gauss sums, including the singular Gauss
sums, that occur as Fourier transforms of $\chi_{\rho}$ with respect to
the characters of $M(n,\Z/p\Z)$.

For the non-singular Gauss sums, the formula of Kondo, namely
Eq. \eqref{kondo-exact} suffices but we need to analyze  the
singular Gauss sums as well.  After collecting some background
material on conjugacy classes and  representations of $GL(2, \F_p)$ in
\S 2, we analyze these singular Gauss sums for $GL(2,\F_p)$ and prove
the main result for them, namely Theorem \ref{theorem:sgs},  in \S 3.
In \S 4, we carry out the analytic part of the proof of Theorem
\ref{main2}, thus completing the proof.

The next sections are on applications.  Theorem
\ref{thm:growth-elliptic} is proved in \S 5 and to obtain the specific
error term, we use the classical Polya-Vinogradov bound together with
a counting argument.  The proof of Theorem \ref{lqprthm} is  given in
\S6.  A natural  idea here would be to first expand the indicator
function of the set $\Omega_{prim}$ in terms of the characters and
then to  apply Theorem \ref{main} and estimate the sum of the Fourier
coefficients. This is done in \S 6.1 after obtaining bounds for the
sum of the Fourier coefficients (see Lemma \ref{sumfc}) and we obtain
a weaker result, namely, Prop. \ref{weaklqpr}. 

Note that the problematic term $O(p^{3+\ep})$ in Prop. \ref{weaklqpr} arises from 
Theorem \ref{gaussest-general} and the bound $d(\rho) \leq p+1$.
In order to
improve upon this, we need to carefully analyze and accordingly
utilize the instances where  the estimate in 
Theorem \ref{gaussest-general}
 can be improved to $O(p^{2+\ep})$.
The one-dimensional representations do not pose a problem, and there is no contribution
from the principal series as their characters vanish on
$\Omega_{prim}$.  The
improvement arises from two crucial observations. One is the striking
fact that  $|\tr
(\rho (A))|\leq 2$ for non-central elements  of  $GL(2, \F_p)$ (see
Prop. \ref{curi}), which allows one to improve the estimate in
Theorem \ref{gaussest-general} by a factor of $p$ when $A$ is
non-singular.  The second observation is that the trivial and the Steinberg
representations are  related. Their contributions can be clubbed
together as the main
term, allowing one to avoid the problems arising from the  
contributions of the singular Gauss sums attached to 
the Steinberg reprsentation which are of order $p^3$. An appeal to
Prop. \ref{thm:growth-elliptic}  finishes the proof of Theorem \ref{lqprthm}. 

The proof of Prop. \ref{thm:growth-elliptic} rests only on the classical Polya-Vinogradov theorem, whereas that of Theorem \ref{lqprthm} makes use of the 
non-abelian versiod developed in this paper. Thus, the proof of Theorem \ref{lqprthm}, involves both the $GL(1)$ and $GL(2)$-versions of the Polya-Vinogradov type theorems. 
  
  Finally, in \S \ref{Perel}, we explain an alternative approach towards the problem of counting primitive elements using older results on exponential sums
  that depend crucially on the work of Weil on the Riemann Hypothesis for curves over finite fields. 

{\bf Acknowledgement.}  This work was started when the second author
visited ISI, Kolkata in March, 2016. Both the authors thank ISI and
TIFR, Mumbai where much of the work was carried out for excellent
working condition. The second author thanks MPIM,  Bonn for two visits
during May of 2018 and 2019, for an excellent working environment
allowing the authors to make progress on these questions. It is a
pleasure to acknowledge J.-M. Deshouillers, \'{E}. Fouvry, E. Ghate, 
H. Iwaniec, F. Jouve, D. Prasad, O. Ramar\'{e}, D.S. Ramana, S. Sen,
S. Varma for their interest, suggestions and encouragement.

\section{Conjugacy classes and representations of $GL(2,\F_p)$}

\subsection{Conjugacy classes in $GL(2,\F_p)$}
Let $p$ be an odd prime. 
We recall the classification of conjugacy classes in $GL(2,\F_p)$
(see \cite{FH}): 

\noindent{\em Central elements.}
The central elements given by scalar matrices. These have order
  dividing $(p-1)$. 

\noindent{\em Non-semisimple classes.}
The non-semisimple elements are conjugate to a matrix of
  the form 
$\bigl(\begin{smallmatrix} x &1\\ 0&x \end{smallmatrix} \bigr)$, 
with $x\in \F_p^*$. The order of these elements divides $p(p-1)$.

\noindent{\em Split semisimple classes.}
The non-central split semisimple elements are those whose characteristic
polynomials have distinct roots in $\F_p$. These are
conjugate to a matrix of the form  
$\bigl(\begin{smallmatrix} x &0\\ 0&y \end{smallmatrix} \bigr)$,  
with $x, ~y \in \F_p^*, ~x\neq y$.  These
elements have order dividing $(p-1)$.

\noindent{\em Elliptic semisimple classes.}  The elliptic (or
non-split) semisimple conjugacy classes are those whose characteristic
polynomials are irreducible over $\F_p$.  Let $\tau$ be a non-square
in $\F_p$, and  $\tau'\in \F_{p^2}$ be a squareroot of $\tau$. The
roots of the characteristic polynomial  can be written as
$\zeta_{x,y}=x+\tau'y$ and $\zeta_{x,y}^p=x-\tau'y$. The matrix
$d_{x,y}=\bigl(\begin{smallmatrix} x & \tau y\\ y & x
\end{smallmatrix} \bigr)$,  with $x, ~y, ~\tau \in \F_p, ~y\neq 0$ is
a representative for the conjugacy class determined by $\{\zeta_{x,y},
~\zeta_{x,y}^p\}$.  These elements have order dividing $(p^2-1)$. 

The action of $\F_{p^2}$ on itself by multiplication gives an
embedding of $\F_{p^2}^*$ into  $GL(2,\F_p)$ and thus a generator for
the cyclic group $\F_{p^2}^*$ yields an element   of  order $(p^2-1)$
in  $GL(2,\F_p)$.    The matrix $d_{x,y}$ is the matrix of the
transformation given by multiplication by  $\zeta_{x,y}=x+ {\tau}'y$
on $\F_{p^2}$ with respect to the basis $(1, \tau')$ of $\F_{p^2}$
over $\F_p$.   The determinant of  $d_{x,y}$ is $
N(\zeta_{x,y})=\zeta_{x,y}^{p+1}$, where $N: \F_{p^2}^*\to \F_p^*$ is
the norm map. 

Let $\Omega_e$ denote the set of elliptic semisimple elements in
$G=GL(2,\F_p)$. The centralizer of an elliptic element $d_{x,y}$ is
the group $\F_{p^2}^*$. Hence the number of elements in the conjugacy
class is $p^2-p$. Since the elliptic classes are parametrized  by
pairs of elements of the form $\{\zeta, \zeta^p\}$, with $\zeta \in
\F_{p^2}^* \backslash \F_p^*$, the number of elliptic conjugacy
classes is $(p^2-p)/2$. Thus the cardinality of $\Omega_e$ is
$(p^2-p)^2/2$. 

Let $\Omega_{prim}$ be the subset  of $\Omega_e$ consisting of elements
of order $p^2 -1$. From the description of the conjugacy classes we
note that these are the elements with maximum order  in $GL(2,\F_p)$
and can be thought of as two-dimensional analogues of primitive roots;
i.e., (elliptic) generators of $\F_{p^2}^*$. The number of such classes is
$(p^2-p)\phi(p^2-1)/2$, where $\phi$ denotes the Euler
$\phi$-function. The proportion of these classes in $G$ is given by, 
\[
\frac{|\Omega_{prim}|}{|G|}=\frac{\phi(p^2-1)(p^2-p)/2}{(p^2-1)(p^2-p)}
=\frac{\phi(p^2-1)}{2(p^2-1)}.\]

\subsection{Irreducible representations of  $GL(2,\F_p)$}\label{irred}
The irreducible complex representations of  $GL(2,\F_p)$ were
classified by Schur. Green (\cite{G}) constructed the irreducible
characters of $GL(n,\F_p)$ parametrized by the 
conjugacy classes in $GL(n,\F_p)$. We recall the classification 
of the  irreducible complex 
representations of  $G=GL(2,\F_p)$ (see \cite{FH}). 

\noindent{\em One dimensional representations.}
The one dimensional representations $U_{\chi}$, corresponding to
  the scalar matrices, defined by 
$U_{\chi}(A)=
  \chi(\deter(A))$, where $\chi$ is character of
  $\F_p^*$. There  are $(p-1)$ isomorphism classes, and 
\begin{equation}\label{char:abelian}
\chi(d_{x,y})=\chi(N(\zeta_{x,y})).
\end{equation}

\noindent{\em Irreducible Principal series.}
Given a subgroup $H$ of a finite group $G$, and a representation
$\theta$ of $H$ on $V$, a model for the 
induced representation $\rho=I_{H}^G(\theta)$ can be taken as follows: 
\begin{equation}\label{eqn:model0}
  I_{H}^G(\theta)=\{f:G\to V\mid f(gh)=\theta(h)^{-1}f(g)\quad \forall
  h\in H\}.
\end{equation}
The group $G$ acts on the left: $(\rho(g_0)f)(g)=f(g_0^{-1}g)$ for 
$g_0, g\in G$. 

Let $P$ (resp. $P'$, $U$, $U'$) denote the subgroups of  $G$
consisting of lower triangular (resp. upper triangular, unipotent lower
triangular, unipotent upper triangular) 
matrices in $GL(2,\F_p)$. 
The principal series representations $I_{\chi, \eta}$ are 
indexed by pairs of  distinct characters $\chi, \eta$ of
$\F_p^*$, and correspond to the non-central split semisimple conjugacy
classes.  Via the exact sequence, 
\[ 1\to U'\to P'\to (\F_p^*)^2\to 1,\]
$\chi\oplus \eta$ defines a representation of $P'$, and $I_{\chi,
  \eta}$
is defined  to be the induced
representation  $I_{P'}^{G}(\chi\oplus \eta)$.  
We have isomorphisms $I_{\chi, \eta}\simeq I_{\eta,
  \chi}$. The dimension of these representations is $p+1$, and the character
of these representations vanish on
the set of elliptic semisimple conjugacy classes. 

\noindent{\em Twists of Steinberg.} 
Given a character $\chi$ of $\F_p^*$, there is a decomposition, 
\[I_{P'}^{G}(\chi\circ \deter)= St_{\chi}\oplus \chi\circ \deter.\]
The Steinberg representation $St$ corresponds to the trivial
character $1_{P'}$ of $P'$. 
The induced representation  $I_{P'}^{G}(1_{P'})$ is the
regular action of $G$ on the space of functions on the 
projective line ${\mathbbm P}^1=G/P'$. Given two
functions $f_1, f_2$ on ${\mathbbm P}^1$, an invariant inner product is,
\[ \la f_1, f_2\ra =\sum_{x\in {\mathbbm P}^1}f_1(x)\overline{f_2(x)}. \]
The  Steinberg $St$ is the orthogonal complement of the trivial
representation in $I_{P'}^{G}(1_{P'})$. The underlying space
$V_{St}$ for the Steinberg is, 
\begin{equation}\label{Stmodel}
 V_{St}=\{ f:{\mathbbm P}^1\to \C\mid \sum_{x\in {\mathbbm P}^1}f(x)=0\}.
\end{equation}
We have $St_{\chi}=St\otimes \chi\circ \deter$.  
These representations correspond to the non-semisimple conjugacy classes. 
The dimension of these
  representations is $p$, and there are  $(p-1)$ representations upto
  isomorphism. The character $St_{\chi}$  on
  an elliptic semisimple element is given by 
\begin{equation}\label{char:steinberg}
\mbox{Tr}\left(St_{\chi}(d_{x,y})\right)=- \chi(N(\zeta_{x,y})).
\end{equation}

\noindent{\em Cuspidal representations.}  The cuspidal representations
$X_{\phi}$ are indexed by characters $\phi$  of $\F_{p^2}^*$
satisfying $\phi\neq \phi^p$. They  are defined by the property  that
the invariants with respect to the subgroup $U'$ is trivial, and
correspond to the elliptic conjugacy classes.  The dimension of these
representations is $p-1$, and there are  $(p^2-p)/2$ distinct cuspidal
representations. The character of $X_{\phi}$ vanishes on the split semisimple
conjugacy classes, and on elliptic conjugacy classes its value is, 
 \begin{equation}\label{char:cuspidal}
\mbox{Tr}\left(X_{\phi}(d_{x,y})\right)=- (\phi(\zeta_{x,y})+ \phi(\zeta_{x,y}^p)).
\end{equation}


\section{Singular Gauss sums}
In this section we compute the trace of  $G(\rho, A)$, where $A$ is a
singular matrix in $M(2,\F_p)$ and prove Theorem \ref{theorem:sgs}.
  We refer to $Tr\left(G(\rho, A)\right)$ as {\em singular Gauss
sums}. When $A$ is the zero matrix,
\[ G(\rho, A)=\sum_{X \in G} \rho(X)=\begin{cases} 0 \quad \mbox{if
      $\rho$ is irreducible, non-trivial,}\\
|G| \quad \mbox{if $\rho$ is trivial.}
\end{cases}
\]
Suppose now $A$ is a non-zero singular matrix. 
For any $Z \in G=GL(2,\F_p)$
\[
 G(\rho, ZAZ^{-1}) =\sum_{X \in G} \rho(X)e_p(tr(ZAZ^{-1}X))=
 \rho(Z)G(\rho, A) \rho(Z^{-1}). 
\]
Therefore, as far as determination of the trace of $G(\rho, A)$ is
concerned, it is enough to consider the matrices $A$ up to conjugacy:  
\begin{description}
\item[Semisimple case] $A_a:= \begin{pmatrix} a & 0\\ 0 &
  0\end{pmatrix},  ~~a\neq 0.$
\item[Nilpotent case]
$N= \begin{pmatrix} 0 & 1\\ 0 &
  0\end{pmatrix}.$
\end{description}

\subsection{A decomposition for the singular Gauss sum}
The calculation of the singular
Gauss sums uses a Bruhat type decomposition of $GL(2,\F_p)$.
\begin{lemma}
Let $P$ (resp. $P'$) and $U$ (resp. $U'$) denote the subgroups of 
lower triangular (resp upper triangular) 
and lower unipotent (resp. upper unipotent)
  matrices in $GL(2,\F_p)$.  Then
\begin{equation}\label{eqn:decomposition}
GL(2,\F_p)=PU'\sqcup  Pw=PU'\sqcup  wP'
\quad\mbox{and}\quad 
GL(2,\F_p)=U'wP'\sqcup  P',
\end{equation}
where $w
=\begin{pmatrix} 0 & 1\\ 1& 0 \end{pmatrix}$. 
\end{lemma}
\begin{proof} The second decomposition is the Bruhat decomposition. 
The first decomposition can be obtained from the
  Bruhat decomposition $GL(2,\F_p)=PwU\sqcup  P$ by multiplying on the
  right by $w$, and using the fact that $wUw=U', ~wPw=P'$. 
\end{proof}
The group $P$ of upper triangular matrices factorizes  
as a product $P =U\times M\times L$,
where \[M= \left\{\begin{pmatrix} 1 & 0\\ 0& m \end{pmatrix}: m \neq 0 \right\} \textnormal{ and }
L= \left\{\begin{pmatrix} l & 0\\ 0& 1 \end{pmatrix}: \ell \neq 0 \right\}.
\]
 We shall write an element $X\in PU'$ as 
\begin{equation}\label{eqn:ulmu'}
X=x_u  x_l x_m x_{u'}, 
\end{equation}
where 
\[x_u = \begin{pmatrix} 1 & 0\\ u& 1\end{pmatrix},\quad
 x_m =\begin{pmatrix} 1 & 0\\ 0 & m\end{pmatrix}, \quad 
 x_l = \begin{pmatrix} l & 0\\ 0& 1 \end{pmatrix} \quad \mbox{and}\quad
x_{u'}=\begin{pmatrix} 1 & u'\\ 0& 1 \end{pmatrix}.\]
Note that
such a representation is unique. We also note that $x_l$ and $x_m$ commute.
Similarly we shall write an element $X \in wP'$ as
$$X= w x_l x_m x_{u'} .$$  

Corresponding to the first decomposition given in the foregoing lemma,
we write
\[
 G(\rho, A)= G_1(\rho, A) +G_2(\rho, A), 
\]
where 
\begin{equation}\label{GSdecomposition}
 G_1(\rho, A)= \sum_{X \in PU'} \rho(X)e\left(\frac{tr(AX)}{p}\right)\quad
\mbox{and}\quad G_2(\rho, A)=\sum_{X \in wP'}
\rho(X)e\left(\frac{tr(AX)}{p}\right).
\end{equation}
We now compute the traces. For $A_a:= \begin{pmatrix} a & 0\\ 0 &
  0\end{pmatrix}$, a semisimple singular matrix, 
\begin{equation}\label{eqn:traceAs}
\tr(A_a x_u  x_l x_m x_{u'})=al\quad \mbox{and} 
\quad \tr(A_a w x_l x_m x_{u'})=0.
\end{equation}
When $N= \begin{pmatrix} 0 & 1\\ 0 &
  0\end{pmatrix}, $ the traces are, 
\begin{equation}\label{eqn:traceAn}
\tr(Nx_u  x_l x_m x_{u'})=ul\quad \mbox{and} 
\quad \tr(Nw x_l x_m x_{u'})=l.
\end{equation}

\subsection{Vanishing criteria for the singular Gauss sums}
Given a representation $\rho: G\mapsto GL(V)$ 
and a subgroup $H$ of $G$, the projection operator $Pr_{H}\in \mbox{End}(V)$
\[
 Pr_{H} (v)= \left(\frac{1}{|H|}\sum_{h \in H} \rho(h)\right) (v), 
\]
maps $V$ to the subspace $V^{H}$ of vectors invariant under $H$. 
The operator satisfies the projection property $Pr_{H}^2=Pr_{H}$. 

The reason for splitting the singular Gauss sums in terms of the
Bruhat decomposition are the following expressions for $G_1$ and
$G_2$ in terms of projection operators: 

\begin{align}
G_1(\rho, A_a)&=\sum_{X \in PU'} e(al/p)\rho(x_u  x_l)\rho(x_m
  x_{u'})=|MU'|\sum_{\substack{u\in \F_p\\ l\in \F_p^*}}
  e(al/p)\rho(x_u  x_l)Pr_{MU'}
\label{eqn:proj1}\\
G_2(\rho, A_a)&=\sum_{X \in P'}\rho(w)\rho(X)=|P'|\rho(w)Pr_{P'}\label{eqn:proj2}\\
G_1(\rho, N)&=\sum_{X \in PU'} e(ul/p)\rho(x_u  x_l)\rho(x_m
  x_{u'})=|MU'|\sum_{\substack{u\in \F_p\\l\in \F_p^*}}
  e(ul/p)\rho(x_u  x_l)Pr_{MU'}
\label{eqn:proj3}\\
G_2(\rho, N)&=\sum_{X \in P'}e(l/p)\rho(w)\rho(x_l)\rho(x_mx_{u'})
=|MU'|\rho(w)\sum_{l\in \F_p^*}e(l/p)\rho(x_l)Pr_{MU'}.\label{eqn:proj4}
\end{align}

As all the above sums involve the projection operator $Pr_{MU'}$, we observe
the following easy consequence:
\begin{proposition}\label{cor:singvanish1}
Let $A$ be a non-zero singular matrix in $M(2,\F_p)$. Suppose 
$\rho$ is a non-trivial irreducible representation of $GL(2,\F_p)$
acting on the space $V_{\rho}$. Then, the singular Gauss sums 
$G(\rho, A)$ vanish if $V_{\rho}^{MU'}=(0)$. 

Further,  if $V_{\rho}^{P'}=(0)$, then  $G_2(\rho, A_a)$ vanishes.

For the trivial representation $1_G$, 
the singular Gauss sums are equal to
$-p(p-1)$. 
\end{proposition}
\begin{proof}
Only the part about the trivial representation needs to be proved. We
have, 
\begin{align*}
G_1(1_G, A_a)&=|MU'|\sum_{\substack{u\in \F_p\\ l\in \F_p^*}} e(al/p)=-|MU'|p\\
&=-p^2(p-1).\\
G_2(1_G, A_a)&=|P'|=p(p-1)^2.
\end{align*}
Hence, $G(1, A_a)=-p^2(p-1)+p(p-1)^2=-p(p-1)$. 

Similarly, 
\begin{align*}
 G_1(1, N)&=|MU'|\sum_{\substack{u\in \F_p\\l\in \F_p^*}}
e(ul/p)=0.\\
\mbox{and} \quad G_2(1, N)&=|MU'|\rho(w)\sum_{l\in \F_p^*}e(l/p)=-|MU'|\\
&=-p(p-1).
\end{align*}
\end{proof}

\subsection{Vanishing of certain singular Gauss sums}
We now classify those irreducible representations of  $GL(2,\F_p)$
whose $MU'$-invariants are non-zero:
\begin{lemma}\label{lemma:MU'zero}
Let $\rho$ be a non-trivial irreducible representation of $GL(2,\F_p)$
acting on the space $V_{\rho}$. Then 
the invariant space $V_{\rho}^{MU'}$ is at most one dimensional. 

If the space $V_{\rho}^{MU'}$ is
non-zero, then $\rho$ is isomorphic either to the Steinberg
representation $St$, or one of the irreducible principal series
representations $I_{\chi, 1}$ with $\chi$ a non-trivial character of
$\F_P^*$. 

The space  $V_{\rho}^{P'}$ is non-zero only for the Steinberg
representation. 
\end{lemma}

\begin{proof}
Given a representation $\eta$ of a subgroup $H$ of a finite group $G$ and a
representation $\rho$ of $G$, Frobenius reciprocity gives an
isomorphism, 
\begin{equation}\label{eqn:frobenius} \mbox{Hom}_G(I_H^G(\eta), \rho)\simeq \mbox{Hom}_H(\eta,
  \mbox{Res}_G^H(\rho)),
\end{equation}
where $\mbox{Res}_G^H(\rho)$ denotes the restriction of $\rho$ to
$H$.

To say that  $V_{\rho}^{MU'}$ is non-zero means that the trivial
representation $1_{MU'}$ occurs in the restriction of $\rho$
to $MU'$. By Frobenius reciprocity, this is equivalent to  $\rho$
being a subrepresentation of $I_{MU'}^G(1_{MU'})$. Inducing
in stages to $P'$ and then to $G=GL(2,\F_p)$ we have, 
\[I_{MU'}^G(1_{MU'})=I_{P'}^G(I_{MU'}^{P'}(1_{MU'})).\]
Let $\chi$ be a character of $\F_p^*$. Consider
$\chi\otimes1_M,$ as a character of $P'$ with its $M$
component being trivial, defined  by the formula
$\chi(x_lx_mx_{u'})=\chi(l)$. By definition, these characters are
trivial on $MU'$. By Frobenius reciprocity applied to $MU'\subset
P'$, these appear as constituents in
$I_{MU'}^{P'}(1_{MU'})$. Since the index of $MU'$ in $P'$ is
$(p-1)$, dimension count yields an isomorphism, 
\[I_{MU'}^{P'}(1_{MU'})=\oplus_{\chi\in
    \hat{L}}\chi\otimes1_M. \]
Hence, 
\[ I_{MU'}^G(1_{MU'})=\oplus_{\chi\in
    \hat{L}}I_{P'}^G(\chi\otimes1_M). \]
From the classification of
irreducible representations of $G$, we obtain
\begin{equation}\label{decomp:indtrivial}
I_{MU'}^G(1_{MU'})=\oplus_{\chi\in
    \hat{L}, \chi\neq 1_L} I_{\chi, 1}\oplus St\oplus
  1_{G}.
\end{equation}
Among these representations, only
$St$ and the trivial representation of $G$ have a non-zero subspace of
  $P'$-fixed vectors.

As a consequence of Frobenius reciprocity and the
fact that the decomposition given by Eq. \eqref{decomp:indtrivial}
is multiplicity free, it follows 
that the space of invariant vectors under $MU'$ is
at most one-dimensional. 
\end{proof}

From Prop. \ref{cor:singvanish1}, Lemma \ref{lemma:MU'zero} and
the classification of representations, 
we conclude the following proposition, proving 
in particular, Part (\ref{vanishingpart}) of Theorem \ref{theorem:sgs}:
\begin{proposition}\label{prop:GSvanishing}
Let $\rho$ be a non-trivial irreducible representation   of
$GL(2,\F_p)$ not isomorphic to the Steinberg or to $I_{\chi, 1}$ for a
non-trivial character $\chi$ of $\F_p^*$. 
Then for any non-zero singular matrix $A$, $G(\rho, A)=0$. 

For a non-trivial irreducible representation   of
$GL(2,\F_p)$,  $G_2(\rho, A_a)$ vanishes unless  $\rho$ is
isomorphic to the Steinberg. 
\end{proposition}

\subsection{Invariant elements in induced representations}
In order to calculate the traces of the singular Gauss sums, we
calculate explicitly the invariant element and the projection to the
space of invariants with respect to the action of $MU'$. 

Given a character $\theta$ of $P'$, a model for the 
induced representation $\rho=I_{P'}^G(\theta)$ is given as follows: 
\begin{equation}\label{eqn:model}
  I_{P'}^G(\theta)=\{f:G\to \C\mid f(gp')=\theta(p')^{-1}f(g)\}.
\end{equation}
The group $G$ acts on the left: $(\rho(g_0)f)(g)=f(g_0^{-1}g)$ for 
$g_0, g\in G$. From the Bruhat decomposition 
$G=U'wP'\sqcup  P'$ a  collection of left coset representatives for $P'$ in
$G$ is given by $U'w$ and the identity element $e$ of $G$. Thus an
element of $\rho$ is determined by its values on $U'w$ and $e$. 

The natural action of $GL(2,\F_p)$ on $\F_p^2$ induces
  a transitive  action of $GL(2,\F_p)$ on the projective 
line ${\mathbbm P}^1(\F_p)$
  consisting of the lines through the origin in  $\F_p^2$.  The
identity coset $eP'$ of $P'$ 
is the isotropy group of the point at `infinity' given by the line
  defined by the vector $\begin{pmatrix} 1\\ 0 \end{pmatrix}$ in
  $\F_p^2$.  The group $U'$ can be identified with its orbit through
  the point `zero' given by 
$\begin{pmatrix} 0\\ 1 \end{pmatrix}=w\begin{pmatrix} 1\\
  0 \end{pmatrix}$. 
This is precisely the affine line $\A^1(\F_p)$. The Weyl element $w$ switches
  the points zero and infinity of ${\mathbbm P}^1(\F_p)$.

It follows that the restriction of $\rho$ to $U'$ splits as a direct
sum of two representations: 
\begin{equation}\label{eqn:resU'}
\rho\mid_{U'}\simeq R_{U'}\oplus 1_{U'}, 
\end{equation}
where $R_{U'}$ is the regular representation of $U'$ on the space of
functions on $U'$. The trivial representation of $U'$ corresponds to
the element of $\rho$ `supported' at infinity.

\begin{lemma}\label{lem:invariantelt}
(a) Let $\chi$ be a non-trivial character of $\F_p^*$ and 
$\rho=I_{P'}^G(\chi\otimes1_M)$ 
be the irreducible representation of $GL(2,\F_p)$
with the model given by Eq. (\ref{eqn:model}). 

Consider the function  $\delta_{in}$ of $G$  defined by, 
\[\delta_{in}(g)=\begin{cases} 0 \quad \quad \quad \mbox{if $g\not\in P'$}\\
\chi(l)^{-1} \quad \mbox{if $g=x_l x_m x_{u'}\in P'$}.
\end{cases}
\]
The function $\delta_{in}$ belongs to the space underlying $\rho$, and
spans the one dimensional
space of $MU'$-invariants of $\rho$.   For
an element $f\in I_{P'}^G(\chi\oplus 1_M)$, 
\begin{equation}\label{eqn:projps}
Pr_{MU'} (f)=\left(\frac{1}{|MU'|}\sum_{x \in MU'} \rho(x)\right)
(f)= f(e)\delta_{in}.
\end{equation}

(b) The space of $MU'$-invariant elements of the Steinberg for the
model given by Eq. (\ref{Stmodel}) is the space spanned by the
function $\delta_{in}=\delta_{\infty}-\frac{1}{p}\delta_{\A^1}$, 
where $\delta_{\infty}$ is the function supported at `infinity' with
value $1$, and $\delta_{\A^1}$ is the characteristic function of  $\A^1$.

Given a function $f\in V_{St}$, the projection to the space of
$MU'$-invariants is given by, 
\[ Pr_{MU'} (f)=f(\infty)\delta_{\infty}-\frac{(\sum_{x\in \A^1}f(x))}{p}\delta_{\A^1}.\]
\end{lemma}
In other words, essentially the lemma says that the invariant element
is the element in the induced model `supported' at infinity, where for
the Steinberg we need to take the projection to the Steinberg of the
function supported at infinity. 

\begin{proof}

(a)  Since $MU'$ respects the Bruhat decomposition $G=U'wP'\sqcup  P'$
it follows that  $\rho(x_mx_{u'})\delta_{in}$ is supported at
the coset $P'$.  From the definition of $\delta_{in}$, 
\[ (\rho(x_mx_{u'})\delta_{in})(e)=\delta_{in}(x_{u'}^{-1}x_m^{-1})
=\delta_{in}(x_m^{-1}x_{u'/m})=1.\]
This proves the invariance of $\delta_{in}$ under the action of
$MU'$. 

To prove the formula for the projection operator, it is sufficient to
show that for any function $f$ supported in the finite part $U'wP'$ of
$G$, the projection is zero. Given an element $x_{v'}\in U'$, 
\begin{align*}
\sum_{m,u'}\rho(x_mx_{u'})(f)(x_{v'}w)&=\sum_{m,u'}f(x_{u'}^{-1}x_m^{-1}x_{v'}w)
=\sum_{m,u'}f(x_{u'}^{-1}x_{mv'}x_m^{-1}w)\\
&=\sum_{m,u'}f(x_{mv'-u'}wwx_m^{-1}w)\\
&=\sum_{m,u'}\chi(m)f(x_{mv'-u'}w)=0.
\end{align*}
(b) For the Steinberg, the calculation is immediate given that it is a
permutation action of $G$ on $G/P'={\mathbbm P}^1$. 

\end{proof}

\subsection{A formula for the  trace}
Equations (\ref{eqn:proj1}, .., \ref{eqn:proj4}) express the 
partial Gauss sums $G_1$
and $G_2$ as operators of the form $TQ$, where $Q^2=Q$ is a projection
operator. For such operators, the trace of $TQ$ is computed by
restricting the action of $T$ to the image of $Q$: 
\begin{lemma}\label{lem:traceprojection}
Suppose $V$ is a finite dimensional vector space and $T, ~Q\in
\mbox{End}(V)$, where $Q^2=Q$. Then 
\[ Tr(TQ)=Tr(QTQ).\]
\end{lemma}
\begin{proof} 
\[Tr(QTQ)=Tr(TQQ)=Tr(TQ). \]
 \end{proof}
We apply this lemma in the context of Lemma \ref{lem:invariantelt} and
the projection operator $Pr_{MU'}$:
\begin{corollary}\label{cor:trace}
With notation as in  Lemma \ref{lem:invariantelt}, let $T$ be an
operator on the space underlying the representation $\rho$. Then,
\[ \tr(T Pr_{MU'})=T(\delta_{in})(e),\]
where $\rho$  is as in Part (a) of  Lemma \ref{lem:invariantelt}.

When $\rho$ is the Steinberg representation, 
 \[ \tr(T Pr_{MU'})=T(\delta_{in})(\infty).\]
\end{corollary}
\begin{proof} The
projection operator $Pr_{MU'}$ projects onto the one dimensional space
of invariants spanned by $\delta_{in}$. Thus the trace is equal to the
multiple of $\delta_{in}$ in  $Pr_{MU'} T Pr_{MU'}(\delta_{in})$. 

For the Steinberg, we observe that this multiple is as given in the
equation. 
\end{proof}
 
\subsection{Proof of Theorem \ref{theorem:sgs}} We now apply Corollary
\ref{cor:trace}, to compute
the traces of the singular Gauss sums for the principal series 
representations and Steinberg. 

\subsubsection{Irreducible principal series: semisimple case}
Suppose $\rho$ is an irreducible principal series  representation 
$I_{\chi,1}$ with $\chi$ a non-trivial character of $\F_p^*$. 
\begin{align*}
 \tr(G(\rho, A_a))&=\tr(G_1(\rho, A_a))=
|MU'|\sum_{u\in \F_p, l\in \F_p^*}e(al/p) \rho(x_ux_l)(\delta_{in})(e)\\
&=|MU'|\sum_{u\in \F_p,l\in \F_p^*}e(al/p)
  \delta_{in}(x_l^{-1}x_u^{-1}).
\end{align*}
Since $U$ acts simply transitively on ${\mathbbm P}^1\backslash \{0\}$, 
only the term $u=0$ corresponds to the point at infinity and 
contributes to the trace. Hence, 
\begin{align*}
 \tr(G(\rho, A_a))&=|MU'|\sum_{l\in \F_p^*}e(al/p) \chi(l)\\
&=p(p-1)\overline{\chi(a)}G(\chi).
\end{align*}

\subsubsection{Irreducible principal series: nilpotent case}
We now consider the case of irreducible principal series and
nilpotent conjugacy class $N$.  We calculate first the $G_1$ term: 
\begin{align*}
 \tr(G_1(\rho, N))&=|MU'|\sum_{u\in \F_p, l\in \F_p^*}e(ul/p) 
(\rho(x_ux_l)(\delta_{in}))(e)\\
&=|MU'|\sum_{u\in \F_p,l\in \F_p^*}e(ul/p)
\delta_{in}(x_l^{-1}x_u^{-1}).
\end{align*}
Again, only the $u=0$ contributes to the trace. The sum becomes, 
\[\tr(G_1(\rho, N))=|MU'|\sum_{l\in \F_p^*} \chi(l)=0.\]

Similarly, the $G_2$-term can be calculated:
\begin{align*}
 \tr(G_2(\rho, N))&=|MU'|\sum_{l\in
                            \F_p^*}e(l/p)\rho(x_l)(\delta_{in})(e)
=|MU'|\sum_{l\in \F_p^*}e(l/p)\delta_{in}(x_l^{-1})\\
&=|MU'|\sum_{l\in \F_p^*}e(l/p)\chi(l)=p(p-1)G(\chi).
\end{align*}
Hence, 
\begin{align*}
 \tr(G(\rho, N))&=\tr(G_1(\rho,N))+\tr(G_2(\rho,
                          N))\\
&=p(p-1)G(\chi).
\end{align*}

\subsubsection{Steinberg: semisimple case}
We consider now the Steinberg representation. By Corollary \ref{cor:trace},
\begin{equation*}
 \tr(G_1(St, A_a))=|MU'|\sum_{u\in \F_p, l\in \F_p^*}e(al/p) 
(\rho(x_ux_l)(\delta_{\infty}-\frac{1}{p}\delta_{\A^1})(\infty)\\.
\end{equation*}
The group $U$ fixes $0$ of ${\mathbbm P}^1$ and acts by translations on
${\mathbbm P}^1\backslash \{0\}$. Hence for the $\delta_{\infty}$ term, 
only $u=0$ contributes non-trivially.  Hence, 
\[
|MU'|\sum_{u\in \F_p, l\in \F_p^*}e(al/p) 
(\rho(x_ux_l)(\delta_{\infty})(\infty)=|MU'|\sum_{l\in \F_p^*}e(al/p)=-p(p-1).\]
Similarly, for the
$\delta_{A^1}$-term, the contribution comes from non-zero $u$.
Taking infinity to be given by the column vector 
$\begin{pmatrix} 1 \\ 0\end{pmatrix}$, the calculation becomes, 
\begin{align*}
|MU'|&\sum_{u\in \F_p, l\in \F_p^*}e(al/p) 
(\rho(x_ux_l)\frac{1}{p}\delta_{A^1}(\infty)=|M|\sum_{u, l\in
  \F_p^*}e(al/p) \delta_{\A^1}\left(\begin{pmatrix} l^{-1} \\
  -u\end{pmatrix}\right)\\
&=|M|\sum_{u, l\in \F_p^*}e(al/p)=-(p-1)^2.
\end{align*}
Hence, 
\[\tr(G_1(St, A_a))=-p(p-1)+(p-1)^2=-(p-1).\]

By Eq. (\ref{eqn:proj2}), the second sum becomes, 
\[
\tr(G_2(St, A_a))=|P'|\rho(w)\delta_{in}(\infty)
=|P'|\delta_{in}(0)=-\frac{|P'|}{p}\delta_{\A^1}(0)
=-(p-1)^2.
\]

Hence 
\begin{align*} 
\tr(G(St, A_a))&= \tr(G_1(St, A_a)) +\tr(G_2(St, A_a))\\
&=-(p-1)-(p-1)^2\\
&=-p(p-1).
\end{align*}

\subsubsection{Steinberg: nilpotent case}
When the conjugacy class of $A$ is nilpotent, we argue as above in the
semisimple case, considering the sum over $u=0$ and $u$ non-zero
separately. From Eq. (\ref{eqn:proj3}) and Corollary
\ref{cor:trace}, $\tr(G_1(St, N))$ is equal to
\begin{align*}|MU'|\sum_{l\in \F_p^*} 
&(\rho(x_l)(\delta_{\infty})(\infty)-
|M|\sum_{u, l\in \F_p^*}e(ul/p)
  (\rho(x_ux_l)(\delta_{\A^1})(\infty)\\
&=|LMU'|-|M|\sum_{u, l\in \F_p^*} e(ul/p)=p(p-1)^2+(p-1)^2\\
&=(p+1)(p-1)^2.
\end{align*}
From Eq. (\ref{eqn:proj4}), the second sum becomes, 
\begin{align*}\tr(G_2(St, N))
&=|MU'|\rho(w)\sum_{l\in \F_p^*}e(l/p)\rho(x_l)\delta_{in}(\infty)\\
&=|MU'|\sum_{l\in\F_p^*}e(l/p)\rho(x_l)\delta_{in}(0)\\
&=-\frac{|MU'|}{p}\sum_{l\in\F_p^*}e(l/p)\rho(x_l)\delta_{\A^1}(0)\\
&=-|M|\sum_{l\in \F_p^*}e(l/p)\\
&=(p-1).
\end{align*}

Hence 
\begin{align*} 
\tr(G(St, N))&= \tr(G_1(St, N)) +\tr(G_2(St, N))\\
&=(p-1)+(p-1)^2(p+1)\\
&=p^2(p-1).
\end{align*}

This proves Theorem \ref{theorem:sgs}. 

\begin{remark}
  It will be interesting to figure out the nature of
  these singular
traces for general $GL(n,\F_p)$. To try to make sense of these values
in terms of the parametrization of the representations 
by the conjugacy classes, we make two definitions:
\begin{definition}
An irreducible representation $\rho$ of $GL(2,\F_p)$ to be of
{\em unit class} if the semisimple part of the conjugacy class
parametrizing it has $1$ as an eigenvalue.
\end{definition}
 
\begin{definition}
The {\em unit multiplicity} $k(\rho)$ of an irreducible representation
$\rho$ is defined to be the  
multiplicity of the eigenvalue $1$ in the 
semisimple part of the conjugacy class
parametrizing it.
\end{definition}
\noindent
The unit multiplicity appears as a `defect' term in 
Kondo's estimate for the non-abelian Gauss sum:  
\[ |g(\rho)| =p^{(n^2-k(\rho))/2}.\]
From the classification given by Theorem \ref{theorem:sgs}, 
we see that the non-trivial unit
class representations of $GL(2,\F_p)$ is isomorphic to either the
trivial or Steinberg or to the prinicipal series
representation $I_{\chi, 1}$ for some non-trivial character $\chi$ of
$\F_p^*$.  
These are precisely the representations that occur in the 
induced representation $I_{MU'}^G(1)$. 
Theorem \ref{theorem:sgs}
says that the singular Gauss sums does not vanish precisely for the 
representations of unit class. 

\end{remark}

\section{Proof of the $GL(2)$ Polya-Vinogradov theorem}

As we have already obtained the bound for Gauss sums, what remains in order to prove Theorem \ref{secondmain} 
is the Fourier analytic part which we develop here.  First we recall some basic facts from Fourier Analysis on finite abelian groups and then we
proceed as in the standard proofs of the classical Polya-Vinogradov Theorem. We consider the case of general $n\times n$ matrices until the 
point when we need to apply the Gauss sum bound. 
\subsection{Fourier analysis on finite groups}
Let $G$ be a finite group. Let $\hat{G}$ denote the set
of isomorphism classes of irreducible complex representations of $G$.  
For $\rho\in \hat{G}$, let $\chi_{\rho}$ denote its character. 
The space of complex valued
functions on $G$ carries an inner product, 
\[ \la f_1, f_2\ra =\frac{1}{|G|}\sum_{x\in
  G}f_1(x)\overline{f_2(x)}, \]
where $f_1, f_1$ are complex valued functions on $G$, and $|G|$
denotes the cardinality of $G$.  With respect to this inner product,
the characters of $G$ form an orthonormal basis for the conjugation invariant functions
on $G$. On the space of functions on
$\hat{G}$, define the inner product
\[ \la \phi_1, \phi_2\ra =\sum_{\rho\in
  \hat{G}}\phi_1(\rho)\overline{\phi_2(\rho)}, \]
where $\phi_1, \phi_1$ are complex valued functions on $\hat{G}$.
For a conjugacy invariant
function $f$ on $G$, its Fourier transform $\hat{f}$ is a function on
$\hat{G}$, defined by $\hat{f}(\rho)=\la f, \chi_{\rho}\ra$. With
these normalizations, the Fourier transform $f\mapsto \hat{f}$ is an
isometry from conjugacy invariant functions on $G$ to functions on
$\hat{G}$.

\subsection{Dual of $M(n, \Z/p\Z)$}
We specialize the foregoing discussion to the case when $G= M(n,
\Z/p\Z)$, where $p$ is a prime number. Denote by $e$ the exponential function
$e(x)= \mbox{exp}(2\pi ix), ~x\in \C$. From the identification of the finite
field $\F_p$ with $\Z/p\Z$, we have an additive character $e_p$  of $\F_p$
given by  $e_p(x)=e(x/p)$.  
Let $n$ be  a positive integer.
For each matrix $A\in M(n, \Z/p\Z)$, consider the character
$\psi_{A}(X)=e_p(\tr(AX))$. We have, 
\begin{lemma} The map $A\mapsto \psi_A$ yields an isomorphism of  $M(n, \Z/p\Z)$
with its dual group $\reallywidehat{M(n, \Z/p\Z)}$.
\end{lemma}
\begin{proof} Since for each non-zero matrix $A\in M(n, \Z/p\Z)$, there exists a
  matrix $X$ with $\tr(AX)\neq 0$, the map   $A\mapsto
  \psi_A$ is injective. Hence the lemma follows by comparing
  the cardinalities. 
\end{proof}

For functions $\phi, \phi':M(n, \Z/p\Z) \longrightarrow \C$, the isometry of
Fourier transform translates to the following Plancherel formula:
\begin{equation}\label{FTiso}
\frac{1}{p^{n^2}}\sum_{A\in
  M(n,\Z/p\Z)}\phi(A)\overline{\phi'(A)}=\sum_{A \in M(n,\Z/p\Z)}
\widehat{\phi}(A)\overline{\widehat{\phi'}(A)}.
\end{equation}

\subsection{A general estimate for box sums} Let $\mathbf{I}$ be an
$n^2$-dimentional matrix interval in $M(n, \Z)$;  i.e., $\mathbf{I}$ is the
Cartesian product of $n^2$ many intervals $I_{ij}, 1\leq i, j \leq n$ for each entry,
where each $I_{ij}$ is an interval in $\Z$.  We may assume
without loss of generality that the length of the  interval $I$ is at
most $p$. Let  $\phi$ be  a complex valued function on $M(n, \Z/p\Z)
$. The following proposition gives an estimate of the general sum 
\[S(\phi, \mathbf{I}) = \sum_{ A \in \mathbf{I}} \phi(A). \]

\begin{proposition}\label{sphi}
Suppose each component interval $I_{ij}$ has length $|I_{ij}|\leq cp$,
where $c>0$ is a constant. Then we have the estimate
\[
 S(\phi, \mathbf{I}) \ll ||\hat{\phi}||_{\infty}p^{n^2}(\log p)^{n^2}
\]
Moreover, for $p\geq 11$, the implied constant can be taken to be ${\left(\frac{c+3}{2}\right)}^{n^2}$. 
\end{proposition} 

By Eq. \eqref{FTiso}, we have
\begin{equation}\label{Parseval}
 p^{-n^2}S(\phi, \mathbf{I}) =\sum_{B \in M(n, \F_p)} \hat{\phi} (B) 
\overline{\widehat{\delta_{\bar{ \mathbf{I}}}}}(B),
\end{equation}
which yields the bound
\begin{equation}\label{Sophi}
|S(\phi, \mathbf{I})| \leq p^{n^2}||\hat{\phi}||_{\infty} \sum_{B \in M(n, \F_p)} 
\left|\hat{\delta}_{\bar{\mathbf{I}}}(B)\right|.
\end{equation}
Hence we need to bound the above sum over $B$ and this will be done in the next few lemmas.

\begin{lemma}
For any real number $\alpha$, we have the bound
\begin{equation}\label{linearsum}
    \sum_{1\leq n\leq N}e(n\alpha) \leq \mbox{min}\left( N, \frac{1}{2||\alpha||}\right),
\end{equation}
where $||\alpha||$ is the distance of $\alpha$ from the
nearest integer.
\end{lemma}
\begin{proof}
This is quite standard. See, e.g., \cite[Chap. 3]{tenlectures}.
\end{proof}

Now we prove a lemma that gives an estimate for ${\widehat{\delta_{\bar{\mathbf{I}}}}}(B)$: 
\begin{lemma}
\[ |{\widehat{\delta_{\bar{\mathbf{I}}}}}(B)| \leq p^{-n^2} \prod_{1\leq i, j\leq n}
min \left(cp, \frac{1}{||b_{ij}/p||}\right). \]

\end{lemma}
\begin{proof}
\begin{align*}
\widehat{\delta_{\bar{\mathbf{I}}}}(B) &=
 \frac{1}{p^{n^2}}\sum_{X \in \mathscr{S}_p}\delta_{\bar{\mathbf{I}}}(X) \psi_X(-B)\\
&=\frac{1}{p^{n^2}}\sum_{X \in \bar{\mathbf{I}}}e\left(\frac{-\tr (BX)}{p}\right).
\end{align*}
Now the sum over $X$ factors as
\[\prod_{i, j}\sum_{x_{ji}\in \bar{\mathbf{I}}_{ji}} e\left(\frac{b_{ij}x_{ji}}{p}\right).
\]
Since for every $(i,j)$ the interval $\bar{\mathbf{I}}_{ji}$ is of length at most $cp$, an application
of Eq. \eqref{linearsum} yields the bound
\begin{align*}
\sum_{x_{ji}\in \bar{\mathbf{I}}_{ji}} e\left(\frac{b_{ij}x_{ji}}{p}\right) &\leq 
min \left(|\bar{\mathbf{I}}_{ji}|, \frac{1}{||b_{ij}/p||}\right)\\
& \leq min \left(cp, \frac{1}{||b_{ij}/p||}\right).
\end{align*}
The lemma follows by taking product over all the entries.
\end{proof}
We now consider the sum over $B$. 
\begin{lemma}\label{sumoverB}
\[ \sum_{B \in M(n, \F_p)} \left|\widehat{\delta_{\bar{\mathbf{I}}}}(B)\right|
 \ll  (\log p)^{n^2}.\]
For $p\geq 11$, the implied constant can be taken to be ${\left(\frac{c+3}{2}\right)}^{n^2}$. 
\end{lemma}
\begin{proof}
By the above lemma,
\[
 \sum_{B \in M(n, \F_p)} \left|\widehat{\delta_{\bar{\mathbf{I}}}}(B)\right|
 \leq \frac{1}{p^{{n^2}}}\sum_{B \in M(n, \F_p)} 
 \prod_{i, j} min \left(cp, \frac{1}{||b_{ij}/p||}\right).
\]  
Since $B$ is varying over the set of all $n \times n$ matrices
over $\F_p$, for each $(i,j)$, $b_{ij}$ varies form $0$ to $p-1$ and
hence the above sum of products can be written as a product of sums as
follows:
\[
\sum_{B \in M(n, \F_p)} 
 \prod_{i, j} min \left(cp, \frac{1}{||b_{ij}/p||}\right)
 =\prod_{i, j}\sum_{0 \leq b_{ij} \leq p-1}min \left(cp, 
\frac{1}{||b_{ij}/p||}\right).
\]
Now we bound the individual sums. We have, 
\begin{align*}
 \sum_{0 \leq b_{ij} \leq p-1}min \left(cp,
   \frac{1}{||b_{ij}/p||}\right) 
&\leq cp+ p\sum_{1 \leq b \leq p-1} \frac{1}{b}\\
&\leq cp+p(1+\log p)\\
& \leq \left(\frac{c+3}{2}\right)p\log p,
\end{align*}
provided that $\log p \geq 2$; i.e., $p \geq 11$.

Hence, for $p \geq 11$, 
\begin{align*}
\sum_{B \in M(n, \F_p)} \left|\widehat{\delta_{\bar{\mathbf{I}}}}(B)\right|
 \leq\left(\left(c+1\right)\log p\right)^{n^2}.
\end{align*}
For smaller primes, a similar bound holds with a different constant. 
\end{proof}
The proof of  Prop.
\ref{sphi} is now clear from Lemma \ref{sumoverB} and Eq. \eqref{Sophi}.

\subsection{Estimate for $\widehat{\chi_{\rho}}$.} Suppose 
$\rho$ is an irreducible complex representation of
$GL(2,\F_P)$. Extend the character $\chi_{\rho}$ of $\rho$ to a
function on $M(n, \Z/p\Z)$ by defining it to be zero on singular
matrices. Then, 
\begin{align*}
 \widehat{\chi_{\rho}}(A)&= \la \chi_{\rho},
 \psi_A\ra=\frac{1}{p^4}\sum_{X\in M(2,
   \Z/p\Z}\chi_{\rho}(X)\overline{\psi_A(X)}\\
&=\frac{1}{p^4}\tr(G(\rho, -A)).
\end{align*}
As a consequence of Theorem
\ref{gaussest-general}, we 
have:
\begin{proposition} \label{chihatestimate}
Let $\rho$ be a non-trivial  irreducible complex representation of
$GL(2,\F_P)$ and $A$ a non-zero matrix. Then
\[ |\widehat{\chi_{\rho}}(A)|\leq d(\rho)p^{-2},\]
where $d(\rho)$ is the dimension of $\rho$. 
\end{proposition}

\subsection{Proof of Theorem \ref{main2}}
We now prove Theorem \ref{main2}.  Let $\rho$ be an irreducible, 
complex representation of $GL(2,\F_p)$. 
In the foregoing notation the sum we want to 
estimate is,
\[S(\chi_{\rho}, \mathbf{I}) =\sum_{A\in \mathbf{I}}\chi_{\rho}(A),\]
where $\mathbf{I}$ is a matrix interval of the form 
\[
I=\prod_{ij} I_{ij},
\]
where, for each pair $(i,j)$ $I_{ij}$ is an interval of length $|I_{ij}|\leq cp.$
By Prop.  \ref{sphi}, we get
\begin{equation}\label{napveqn2}
S(\chi_{\rho}(A), \mathbf{I}) \leq 
||\widehat{\chi}_{\rho}||_{\infty}p^{4}\left((\left(\frac{c+3}{2}\right)\log p\right)^{4},
 \end{equation}
and by Theorem \ref{chihatestimate}, 
\[  ||\widehat{\chi}_{\rho}||_{\infty}\leq p^{-2}d(\rho).
\]
This proves Theorem \ref{main2}.

\section{Growth of elliptic elements: Proof of Theorem
  \ref{thm:growth-elliptic}}
In this section, we give 
an estimate for the function $S(\delta_{\Omega_e}, x)$ that countins the
number of integer matrices  of height up to $x$ that reduce to
elliptic elements modulo $p$. 
In other words, we need to count integer matrices of height up to $x$ for which the
characteristic polynomials are irreducible over $\F_p$; i.e.,
integer matrices $ \begin{pmatrix} a & b\\ c &
  d\end{pmatrix}$ of height up to $x$ such that
$(\tr)^2 -4(\deter)=(a-d)^2+4bc$ is not a quadratic residue modulo $p$.  

Let $\chi$ denote the Legendre symbol modulo $p$ and let $\Delta$ denote
the collection of elements in $M(2,\F_p)$ that have
characteristic polynomials with discriminant divisible by $p$. 
Consider the sum
\begin{equation}\label{SS}
 S = \frac 1 2 \mathop{\sum\sum\sum\sum}_{0\leq |a|,|b|,|c|,|d|\leq
      x}
\left\{ 1-\chi\left((a-d)^2+4bc) \right)     \right\}. 
\end{equation}
When $p$ divides $\deter(A)$, then the discriminant is always a square
modulo $p$. Hence, 
\begin{equation}\label{SSS}
S=S(\delta_{\Omega_e}, x)+\frac{1}{2}S(\delta_{\Delta}, x).
\end{equation}
Now, from Eq. \eqref{SS},
\begin{equation}\label{S}
S=\frac 1 2 (2[x]+1)^4 - \frac 1 2 S',
\end{equation}
where
\[
 S'=\mathop{\sum\sum\sum\sum}_{0\leq|a|,|b|,|c|,|d|\leq x}
 \chi((a-d)^2+4bc).  
 \]   
When $p$ divides $b$, then  $\chi((a-d)^2+4bc)$ is identically
$1$, unless $a\equiv d \modp$ when it vanishes. 
Thus the contribution of terms with $p|b$ to $S'$ is:
\[ (2x/p+O(1))(2x)^3+O(x^3/p)-(2x/p+O(1))(2x)^2(2x/p+O(1))=
  16\left(\frac{1}{p}-\frac{1}{ p^2}\right)x^4+O(x^3).\]
When $b$ is invertible in $\F_p$, we 
pull it out in order to obtain a sum over $c$ varying in an interval which can
be estimated by the classical Polya-Vinogradov bound \eqref{pv}. 
The sum over the other three variables is bounded trivially. Thus the contribution of terms
with $b\not\equiv 0 \modp$ is
\begin{align*}
&\mathop{\sum\sum\sum\sum}_{\substack{0\leq |a|,|b|,|c|,|d|\leq
      x\\ b \not\equiv 0 \modp}} \chi((a-d)^2+4bc)\\
 &= \mathop{\sum\sum}_{0<|a|, |d|\leq x} \sum_{0<|b|\leq x, b\not \equiv 0 \modp}\chi(4b)\sum_{0<|c|\leq x}\chi((4b)^{-1}(a-d)^2+c)\\
 &\leq  \mathop{\sum\sum}_{0\leq |a|, |d|\leq x} \sum_{0\leq |b|\leq
   x, b\not \equiv 0 \modp}|\chi(4b)|\left|\sum_{0<|c|\leq x}\chi((4b)^{-1}(a-d)^2+c)\right|\\
  &\ll  \mathop{\sum\sum\sum}_{0<|a|, |b|, |d|\leq x} \sqrt{p}\log p\\
  &\ll x^3\sqrt{p}\log p.
\end{align*}
Hence, 
\begin{equation}\label{eqn:S}
 S= 8x^4-\frac{8x^4}{p}+ \frac{8x^4}{p^2}+ O( x^3\sqrt{p}\log p).
\end{equation}
It remains to estimate $S(\delta_{\Delta}, x)$ which is the content of the next Lemma. 
\begin{lemma}\label{countingdisc0}
The number of matrices  $ \begin{pmatrix} a & b\\ c &
  d\end{pmatrix}$ of height up to $x$ and with $(a-d)^2+4bc \equiv 0 \modp$ 
  is $\frac{16x^4}{p} + O( x^3)$.
\end{lemma}
\begin{proof}
We need to count $4$-tuples $(a, b, c, d)$ such 
that $-x\leq a, b, c, d \leq x$ and $(a-d)^2 \equiv -4bc \modp$. 
First we note that the number of integers in the interval 
$[-x,x]$ is $2[x]+1=2x+O(1)$ 
and the number of integers in this interval that are divisible by $p$
or lies in a fixed residue class modulo $p$ 
is $2[x/p]+1=2x/p+O(1)$. 
The number of pairs $(a,d)$ with $a \equiv d \modp$ is, 
therefore, 
\[(2x+O(1))(2x/p+O(1))=4x^2/p+O(x),\] 
and for each such
a pair, the number of possible pairs $(b,c)$, i.e., 
with the property $bc \equiv 0 \modp$ is 
\[2(2x-2x/p+O(1))(2x/p+O(1)) +(2x/p+O(1))^2= 4x^2(2/p -1/p^2)+O(x).\]
On the other hand, the number of pairs $(a,d)$ with $a \not\equiv d
\modp$ 
is \[(2x+O(1))(2x-2x/p+O(1))=4x^2(1-1/p)+O(x).\] 
For each such pair, fixing any $b 
\not \equiv 0 \modp$ will determine $c$ modulo 
$p$. Hence, for each pair $(a,d)$ with $a \not\equiv d \modp$ 
there is a total of 
\[(2x-2x/p+O(1))(2x/p+O(1))=4x^2(1/p-1/p^2)+O(x)\]
 many pairs $(b,c)$. 
Hence the total number we want is 
\begin{align*}
(4x^2/p+&O(x))(4x^2(2/p
          -1/p^2)+O(x))+(4x^2(1-1/p)+O(x))(4x^2(1/p-1/p^2)+O(x))\\
&=16x^4/p +O(x^3).
\end{align*} 
\end{proof}
From Eq. \eqref{eqn:S}, \eqref{SSS}
and Lemma \ref{countingdisc0}, we have, 
\begin{align*}
S(\delta_{\Omega_e},x)& =S-\frac{1}{2}S(\delta_{\Delta}, x)\\
&=8x^4-\frac{8x^4}{p} +\frac{8x^4}{p^2}+ O( x^3\sqrt{p}\log p)-
  \frac{8x^4}{p} + O( x^3)\\
&=8\left(1-\frac{2}{p}+\frac{1}{ p^2}\right) x^4 + O(x^3 \sqrt{p}\log p).
\end{align*}
This proves Prop. \ref{thm:growth-elliptic}.

\section{Growth of primitive elements: 
Proof of Theorem \ref{lqprthm}}

In this section our principal interest is in the elliptic semisimple
conjugacy classes (see \S 2.1).   
Our goal is to count
integer matrices of height up to $X$ that reduces to a primitive
element modulo $p$.  

\subsection{Fourier expansion of $\delta_{\Omega_{prim}}$}
In order to estimate $S(\delta_{\Omega_{prim}}, x)), x)$, we begin by following the
method given in \S 1.4. First we expand the characteristic function of
$\delta_{\Omega_{prim}}$  in a finite Fourier series. Denoting $c_{\chi_{\rho}}$ 
by $c_{\rho}$ for ease of notation, we write
\begin{equation}\label{Fourier}
\delta_{\Omega_{prim}}=\sum_{\rho\in \hat{G}} c_{\rho}\cdot \chi_{\rho},
\end{equation}
where $\rho$ varies over the set of irreducible representations
of $G$ and the Fourier coefficients $c_{\rho}$ are given by 
\[
c_{\rho}=\la \delta_{\Omega_{prim}}, \chi_{\rho}\ra =\frac{1}{|G|}
\sum_{\omega \in \Omega_{prim}} \overline{\chi_{\rho}(\omega)}. 
\]

Let $T$ be the collection of  conjugacy classes consisting of primitive element in $G$. 
Each conjugacy class $t\in T$ is of size $(p^2-p)$ and is
defined by a pair $\{\zeta_t, ~\zeta_t^p\}$, where $\zeta_t$ generates 
$\F_{p^2}^*$. Also, recall that
$|G|=(p^2 -p) (p^2 -1)$. Thus we have the following 
formula for the Fourier coefficients:

\begin{equation}\label{fcformula}
  c_{\rho}=\frac{1}{p^2 -1}\sum_{t\in T} \overline{\chi_{\rho}(\zeta_t)}.
\end{equation}

The next proposition gives estimates for the Fourier coefficients 
for different types of characters. 

\begin{proposition}\label{fcdetail}
 (i) For the one-dimensional representation $\rho=U_{\eta}$, where 
$\eta:\F_p^{\ast}\mapsto \C^{\ast}$ is a character,
 \begin{equation}
  c_{U_{\eta}}=\frac{1}{2}\sum_{\substack{d|p^2 -1\\ \text{ord}(\eta)|d}} \frac{\mu(d)}{d};
\end{equation}
in particular, for the trivial character $1_G$, 
the corresponding Fourier coefficient is given by
\begin{align*}
 c_1=c_{1_G}=\frac{|\Omega_{prim}|}{|G|}.
\end{align*}
(ii) For the Steinberg representation $St$ and its 
twists by characters $St_{\eta}$, we have 
 \begin{equation}
  c_{U_{\eta}}=-c_{St_{\eta}}=-\frac{1}{2}\sum_{\substack{d|p^2 -1\\ \text{ord}(\eta)|d}} \frac{\mu(d)}{d}.
\end{equation}
(iii) For the principal series representation $I_{\chi, \eta}$,
\begin{equation}
 c_{I_{\chi, \eta}}=0.
\end{equation}
(iv) For the cuspidal representation $X_{\phi}$,
\begin{equation}
 c_{X_{\phi}}=\sum_{\substack{d|p^2-1\\ \text{ord}(\phi)|d}}\frac{\mu(d)}{d}.
\end{equation}

\end{proposition}

Before proving this, we  recall a lemma expressing the characteristic function of the
set of generators of a cyclic group in terms of characters of the group
(see, e.g., \cite[Eq. (8.5.3), page 302]{Sh}).
\begin{lemma} \label{shapiro}
Let $m$ be a
  natural number and let $C_m$ be the cylic group of order $m$. Let
  $P$ be the subset consisting of generators of  $C_m$. Then 
\begin{equation}\label{charprimitive}
\delta_{P}=\sum_{d|m}\frac{\mu(d)}{d}\sum_{\chi^d=\chi_0}\chi, 
\end{equation}
where $\chi: C_m\to \C^*$ are characters of $C_m$, and $\chi_0$ is the
trivial character.   

\end{lemma}

\begin{proof}
We work with $C_m\simeq \Z/m\Z$. A set of representatives for $P$
is given by the natural numbers $n$ up to $m$ and coprime to $m$. From
the properties of M\"{o}bius $\mu$-function, 
\[ \delta_{P}(n) =\sum_{d|(n,m)}\mu(d).  \]
Let $\xi_d$ be the indicator function: 
\[\xi_d(n)=\begin{cases} 1 \quad \mbox{if} ~d|n\\
0\quad \mbox{otherwise}.
\end{cases}
\]
From orthogonality of characters, 
\[\xi_d(n)=\frac{1}{d}\sum_{\chi^d=\chi_0}\chi(n).\]
Hence, 
\[  \delta_{P}(n)=\sum_{d|m}\mu(d)
\xi_d(n)=\sum_{d|m}\frac{\mu(d)}{d}
\sum_{\chi^d=\chi_0}\chi(n).\]
\end{proof}

Now we prove Prop. \ref{fcdetail}.
\begin{proof}
Let $N: \F_{p^2}^* \rightarrow \F_p$ be the norm map. Then, by \eqref{fcformula}, 
 \begin{align*}
c_{U_{\eta}}&=\frac{1}{p^2 -1}\sum_{t \in T}\overline{\eta}(N(\zeta_t))\\
&=\frac{1}{2(p^2-1)}\sum_{\la \zeta \ra =\F_{p^2}^*} \overline{\eta}(N(\zeta)),
\end{align*}
where the factor $1/2$ is to account for the fact that the same conjugacy class is generated by both $\zeta$ and ${\zeta}^p$.
By Lemma \ref{shapiro}, 
\begin{align*}c_{U_{\eta}}&=\frac{1}{2(p^2-1)}\sum_{\zeta \in \F_p^*} 
\Bigl(\sum_{d|p^2 -1} \frac{\mu(d)}{d}
\sum_{\substack{\chi \in \widehat{\F_{p^2}^*}\\{\chi}^d=\chi_0}}
\chi(\zeta)\Bigr)\overline{{\eta \circ N}}(\zeta)\\
&=\frac{1}{2(p^2-1)}\sum_{d|p^2 -1} \frac{\mu(d)}{d}\sum_{{\chi}^d=\chi_0}
\sum_{\zeta \in \F_p^*} (\chi\,\overline{\eta \circ N})(\zeta)\\
&=\frac{1}{2}\sum_{d|p^2 -1} \frac{\mu(d)}{d}\sum_{{\chi}^d=\chi_0}\sum_{\zeta \in \F_p^*} \delta_{\chi=\eta \circ N}\\
&=\frac{1}{2}\sum_{\substack{d|p^2 -1\\ \text{ord}(\eta)|d}} \frac{\mu(d)}{d},
\end{align*}
by orthogonality of characters and the observation that $\text{ord}(\eta \circ N)=\text{ord}(\eta)$. Note that 
the
condition that $\eta^d$ is the trivial character translates to the
condition that the order of $\eta$ divides $d$. This proves part (i).

Now we consider part (iii). The character of a representation induced from
the Borel subgroup (say, upper triangular matrices $P'$) in $G$ is supported on the
conjugacy classes which intersect $P'$. By definition, the elliptic
classes cannot be conjugated into $P'$. This proves (iii). 

Part (ii) follows from the fact that 
$Ind_{P'}^G({\eta\oplus \eta}_{P'})=\mathbbm{\eta}\circ \deter\oplus
St_{\eta}$, where $\eta$ is a character of $\F_p^*$, and $\eta\oplus
\eta$ is considered as a character of $P'$ via the projection $P'\to
\F_p^*\oplus \F_p^*$. Hence $c_{U_{\eta}}=-c_{St_{\eta}}$. 

Now we prove part (iv).
We have 
 \[c_{X_{\phi}}=\frac{1}{p^2 -1}\sum_{t\in T} -
 (\overline{\phi(\zeta_{t})}+ 
\overline{\phi(\zeta_{t}^p)})=-\frac{1}{(p^2-1)}\sum_{\zeta }
\overline{\phi(\zeta)}, \]
where the last sum runs over all generators $\zeta$ of $\F_{p^2}^*$. 
From Lemma \ref{shapiro}, 
\[ c_{X_{\phi}}=-\frac{1}{(p^2-1)}\sum_{\zeta\in\F_{p^2}^* }
\sum_{d|p^2-1}\frac{\mu(d)}{d}\sum_{\chi^d=\chi_0}\chi\overline{\phi}(\zeta).\]
Interchanging the order of summation, we get
\[ c_{X_{\phi}}=-\frac{1}{(p^2-1)}\sum_{d|p^2-1}\frac{\mu(d)}{d}
\sum_{\chi^d=\chi_0}\sum_{\zeta\in\F_{p^2}^* }\chi\overline{\phi}(\zeta).\]
By orthogonality, the last sum is zero unless $\chi=\phi$ and this proves (iv).
\end{proof}

We give now an estimate for the sum of the Fourier coefficients.
\begin{lemma}\label{sumfc}

We have the estimates
 \[(i) \sum_{\eta} |c_{U_{\eta}}| \leq \tau(p^2 -1);\]
 \[
  (ii) \sum_{\eta} |c_{St_{\eta}}| \leq \tau(p^2 -1);
 \]
 \[(iii) \sum_{X_{\phi}} |c_{X_{\phi}}| \leq \tau (p^2 -1).\]
 Here $\tau(n)$ denotes the number of divisors of $n$.
\end{lemma}

\begin{proof}
 Note that (ii) follows from (i) because $c_{U_{\eta}}=-c_{St_{\eta}}$.
 For part (i), we partition the sum according to the orders 
of the characters $\eta$ and apply the above proposition and 
estimate the sum as follows:
 
 \begin{align*}
  \sum_{\alpha} |c_{U_{\alpha}}|& \leq \frac{1}{2}\sum_{m|p -1}
                                  \phi(m) 
\sum_{\substack{d|p^2-1\\ m|d}}\frac{1}{d}\\
  &=\sum_{d|p^2 -1} \frac{1}{d} \sum_{m|(d, p-1)} \phi(m)\\
  &\leq \sum_{d|p^2 -1} \frac{1}{d} \sum_{m|d} \phi(m)\\
 & =\tau(p^2 -1),
 \end{align*}
where $\phi$ above denotes the Euler $\phi$-function and we have used elementary result $\sum_{m|d}\phi(m)=d$.
 For (iii) we recall that cuspidal representations are parametrized by
characters $\phi$ of ${\F}^{\ast}_{p^2}$ satisfying $\phi \neq
{\phi}^p$.  Suppose $\phi_1$ is a generator of the group of all
characters of ${\F}^{\ast}_{p^2}$. Then  ${\phi_1}^j$ for $
j=1,2,\cdots,p^2 -1$ are all the characters. For estimating the sum in
question, we first enlarge the set  to include all the  characters and
then divide the sum according to the order of the characters. Note
that the number of characters of order $m$ is $\phi(m)$. Thus we
obtain,
 \begin{align*}
  \sum_{X_{\phi}} |c_{X_{\phi}}|& \leq \sum_{m|p^2 -1} \phi(m) 
\sum_{\substack{d|p^2-1\\ m|d}}\frac{1}{d}\\
  &=\sum_{d|p^2 -1} \frac{1}{d} \sum_{m|d} \phi(m),\\
  &=\tau (p^2 -1).
  \end{align*}
\end{proof}

\subsection{Application of the $GL(2)$ Polya-Vinogradov estimate}

At this stage a direct application of Theorem 
\ref{main} and Lemma \ref{sumfc} easily gives us the following:
\begin{proposition}\label{weaklqpr}
\begin{equation}\label{weak}
S(\delta_{\Omega_{prim}}, x)=\frac{8\gamma_p\phi (p^2 -1)}{(p^2 -1)}x^4+O(x^3)
+O\left(p^{3+\ep}\right),
\end{equation}
where $\gamma_p= 1-1/p-1/p^2+1/p^3$.
 \end{proposition}
\begin{proof}
 We are interested in the sum
\begin{equation}
S(\delta_{\Omega_{prim}}, x):=\sum_{h(A) \leq x} \delta_{\Omega_{prim}} (\overline{A}),
\end{equation}
which, after an application of \eqref{Fourier} and 
interchange of summation, becomes
\begin{equation}\label{AOC}
\sum_{\rho} c_{\rho} \sum_{h(A) \leq x} \chi_{\rho}(A),
\end{equation}
from which we isolate the contribution of the trivial character. Thus we obtain
\begin{align*}
S(\delta_{\Omega_{prim}}, x)&= \frac{|\Omega_{prim}|}{|G|}\sum_{h(A) \leq x} \chi_1(A)+
 \sum_{\chi \neq \chi_1} c_{\rho}\sum_{h(A) \leq x} \chi_{\rho}(A)\\
 &=\frac{|\Omega_{prim}|}{|G|}16\gamma_p x^4+O(x^3)+O\left(p^3(\log p)^4 \sum_{\rho} |c_{\rho}|\right)\\
 &=\frac{8\gamma_p\phi (p^2 -1)}{(p^2 -1)}x^4 +O(x^3)+ \left(p^{3+\ep}\right),
\end{align*}
where we have estimated the sum over $A$ 
for non-trivial characters by  Theorem \ref{main}, we have appealed  to Lemma \ref{countingns}
proved below for the sum corresponding to the trivial character, we have 
applied  Lemma \ref{sumfc} for estimating the sum over 
Fourier coefficients, and finally we have applied
the standard bounds:
 $\tau(n),~\log n =O\left(n^{\ep}\right)$ for any $\ep >0$.
\end{proof}

The contribution of the trivial character is given by the following
lemma: 
 \begin{lemma}\label{countingns}
\begin{equation}\label{mainterm-trivialchar}
\sum_{h(A)\leq x} \chi_1(A)=16\gamma_p x^{4}+O(x^{3}),
\end{equation}
where $\gamma_p=1-\frac1p-\frac1{p^2}+\frac1{p^3}.$
\end{lemma}
\begin{proof}
 First we note that
\begin{equation*} 
\sum_{h(A)\leq x} \chi_1(A) = \# \{A \in M(2, \Z): \deter(A)\not\equiv 0 \modp, h(A)\leq x \}
\end{equation*} We count the complimentary set, i.e.,  matrices of
height up to $x$  that are singular modulo $p$ and this amounts to
counting $4$-tuples $(a,b,c,d)$ such that $-x \leq a, b, c, d \leq x$
and $ad-bc \equiv 0 \modp$. An elementary argument as in the proof of 
Lemma \ref{countingdisc0} shows that this number is
\begin{align*}
(4x^2(1-&1/p)^2+O(x))(4x^2(1/p-1/p^2)+O(x))+(4x^2(2/p-1/p^2)+O(x))^2\\
&=16x^4(1/p+1/p^2-1/p^3)
+O(x^3).
\end{align*}
 Upon subtracting this from $ (2[x]+1)^4=16x^4+O(x^3)$, the total number of
matrices of height up to $x$, the lemma follows.
\end{proof}

\subsection{ First steps towards the proof of Theorem \ref{lqprthm}}
In order to prove Theorem \ref{lqprthm}, we need to improve upon the term 
$O(p^{3+\ep})$ in Eq. \eqref{weak} above to
$O(p^{2+\ep})$. The estimate $O\left(p^{3+\ep}\right)$ arises from the estimate 
$tr(G(\rho, A))\leq d(\rho)p^2$ for the
Gauss sums.
Below we make a deeper analysis of the Gauss sums depending on whether $A$ is singular
or non-singular and also depending on what type of representation $\rho$ we have.

Recall that we have (see Eq. \eqref{AOC})
\begin{equation} \label{AOC2}
S(\delta_{\Omega_{prim}}, x)=
\sum_{\rho} c_{\rho} \sum_{h(A) \leq x} \chi_{\rho}(A)
\end{equation}
The first observation is that
in the Fourier expansion of $\delta_{\Omega_{prim}}$ given by
Eq. \eqref{Fourier}, the irreducible principal series do not
occur as $c_{\rho}=0$ for these representations (see Prop. \ref{fcdetail}).
Also,
for the representations $U_{\eta}$ where $\eta$ is a non-trivial character of $\F_p^{\ast}$, we note that the dimension 
$d(U_{\eta})=1$ and hence Theorem \ref{main} gives the bound 
\begin{equation}\label{onedim}
 \sum_{h(A) \leq x} \chi_{U_{\eta}}(A) \ll  p^2 (\log p)^4,
\end{equation}
which is good enough for our purpose.
Therefore, it is enough to consider the trivial representation, the Steinberg representation $St$, the  non-trivial twists of $St$, and
the cuspidal representations $X_{\phi}$. 
Now observe that if $A$ is non-singular, by Eq. \eqref{equivariance}, 
\[\tr(G(\rho, A))=g(\rho)\tr(\rho(A^{-1}),\]
where $|g(\rho)|\leq p^2$. A {\em striking fact} about the values of
irreducible characters of $GL(2,\F_p)$  that 
can be read off the character table for $GL(2, \F_p)$ 
(see \cite[Page 70, Section 5.2]{FH}) is  the following:
  \begin{proposition} \label{curi}
  Suppose $A$ is a $2\times 2$ integer matrix that reduces modulo $p$ to a non-singular matrix which
  is not central. Then for any non-trivial representation $\rho$ of $G$, we have the bound
\begin{equation} 
|\widehat{\chi_{\rho}}(A)|\leq 2p^{-2}.
\end{equation}
\end{proposition}
This suggests that we should isolate the contribution of the scalar matrices after an application of the Plancherel formula
\begin{equation}\label{plan}
\sum_{h(A) \leq x} \chi_{\rho} (A) = p^4\sum_{B \in M(n,\F_p)} \widehat{\chi_{\rho}}(B) \overline{\widehat{\delta_{\bar{\mathbf{I}}}}(B)},
\end{equation}
where $\mathbf{I}$ is the interval 
\[
\mathbf{I}=\{A \in M(2, \Z): h(A) \leq x\}.
\]
Accordingly, we subdivide the resulting  sum over $B$ into three parts: (i)  over singular matrices, 
(ii) over scalar non-singular matrices and (iii) over non-singular matrices that are not scalar. 
However, we do this only for the cuspidal representation and the non-trivial twists of the Steinberg representation.
We treat the trivial and the Steinberg representation together in \S6.5 as they both contribute to the main term.
\subsection{Cuspidal representations and non-trivial twists of the Steinberg representation}
The result we want to prove here is:
\begin{proposition}\label{others}
 Suppose $\rho$ is either a cuspidal representation $X_{\phi}$ 
or a non-trivial twist of the Steinberg representation $St_{\eta}$.
 Then we have the bound
 \[ 
\sum_{h(A) \leq x} \chi_{\rho}(A) \ll x^2 p\log p +p^2(\log p)^4
\]
 \end{proposition}
 \begin{proof}
We apply Eq. \eqref{plan} and
and split the sum on the right hand side into three parts as described at the end of the previous subsection.
The contribution of part (i) is zero by part (1) of Theorem \ref{theorem:sgs}. 
For part (iii), i.e., when  $B$ is non-singular and not scalar, 
we have the bound $\widehat{\chi_{\rho}}(B)  \ll p^{-2}$ 
by
Prop. \ref{curi}. Also, recall that by Lemma \ref{sumoverB}
we have the bound 
\[
\sum_{B \in M(n,\F_p)} {\widehat{\delta_{\bar{\mathbf{I}}}}}(B) \ll(\log p)^4.
\]
This gives the bound $O(p^2(\log p)^4)$ for 
the sum over non-singular and non-scalar matrices. \\
For part (ii), we need to consider  the sum over  non-singular scalar
matrices for characters coming   from $St_{\eta}$ and $X_{\phi}$.  For
$St_{\eta}$, its character takes the value $p \eta (a^2)$, and for
$X_{\phi}$, its character takes the value $(p-1)\phi (a)$ on the
central elements $ \begin{pmatrix} a & 0\\ 0 & a\end{pmatrix}$. 
By $I$ we will denote the identity matrix in $GL(2, \F_p)$ and by $X$
we will denote a variable matrix  
$ \begin{pmatrix} x_{11} & x_{12}\\
x_{21} & x_{22}\end{pmatrix}$.
We recall that (see Equations \eqref{schur} and \eqref{equivariance})
for non-singular $B$,
\[
\widehat{\chi_{\rho}}(B)= g(\rho)\chi_{\rho}(B^{-1}).
\]
Therefore, the sum we need to estimate is
\begin{align*}
p^4\sum_{B=bI, b \not\equiv 0 \modp}\widehat{\chi_{\rho}}(B) 
\overline{\widehat{\delta_{\bar{\mathbf{I}}}}(B)}
=p^{-4}g(\rho) \sum_{b\in \F_p^{\ast}}\chi_{\rho}(b^{-1}I)
\sum_{X \in \bar{\mathbf{I}}} e\left(\frac{b(x_{11}+x_{22})}{p}\right),
\end{align*}
for $\rho =X_{\phi}$ or $St_{\eta}$.
 First we consider the case of cuspidal representations $X_{\phi}$ attached to a 
 character $\phi:\F_{p^2}^*\to \C^*$ satisfying $\phi\neq \phi^p$.
The above sum becomes
\begin{align*}
p^{-4}g(X_{\phi})& \sum_{b\in \F_p^{\ast}}(p-1)\overline{\phi(b)}
\sum_{X \in \bar{\mathbf{I}}} e\left(\frac{b(x_{11}+x_{22})}{p}\right).
\end{align*}
Now we  factor the above  exponential sum and the sums involving $x_{11}$ and $x_{22}$ are estimated by Lemma
\ref{linearsum}, while the sums over $x_{12}$ and $x_{21}$ are  bounded trivially. Thus the above sum is
\begin{align*} 
p^{-4}(p-1)g(X_{\phi})&\sum_{b\in \F_p^{\ast}}\phi(b)\sum_{x_{12}}\sum_{x_{21}}\sum_{x_{11}} e\left(\frac{-bx_{11}}{p}\right) 
\sum_{x_{22}}e\left(\frac{-bx_{22}}{p}\right)\\
&\ll p^{-4}(p-1)p^2 x^2 \sum_{b\in \F_p^{\ast}}|\phi(b) |||b/p||^{-2}\\
&\ll x^2 p \log p,
\end{align*}
where we have used \eqref{kondo-exact} to bound $g(X_{\phi})$.\\
For the characters associated to the representations 
of the type $St_{\eta}$, the treatment is similar. In this case,
$
 \chi_{St_{\eta}}(bI)=\eta (b^2)={\eta}^2 (b)
$
and we obtain the sum
\[
 p^{-4}(p-1)g(St_{\eta})\sum_{X \in \bar{\mathbf{I}}}\sum_{b\in \F_p^{\ast}}{\eta}^2 (b)e\left(\frac{-b(x_{11}+x_{22})}{p}\right).
\]
Proceeding as before we find that this sum is also $O(x^2p\log p)$.
\end{proof}

\subsection{The main term} 
We still have to consider the trivial representation $1_G$ and the Steinberg representation $St$. They are `closely
related', in that they are the components of the representation
parabolically induced from the trivial representation of the Borel subgroup. 
The character values of the trivial and the Steinberg represention
are equal on split semisimple conjugacy classes, and equal but of
opposite sign at the elliptic semisimple conjugacy classes.
This suggests that not just the trivial character, but both the trivial character and the Steinberg character contribute to the main term. This
is the reason we have postponed the treatment of these two representations thus far and we shall now analyze their contribution.

From Eq. \eqref{AOC2}, we write
\begin{equation}\label{lang}
S(\delta_{\Omega_{prim}}, x)= c_1\sum_{h(A) \leq x} \chi_1(A)+c_{St}\sum_{h(A) \leq x} \chi_{St}(A)+
 \sum_{\rho}c_{\rho}\sum_{h(A) \leq x} \chi_{\rho}(A),
\end{equation}
where we recall that $c_1$ is the Fourier coefficient for the trivial representation; i.e., $c_1=c_{1_G}$ and 
$\rho$ runs over  representations that are not isomorphic to $1_G$ or to $St$.
We recall that by part (ii) of  Prop. \ref{fcdetail},
$-c_{St}=c_1=|\Omega_{prim}|/|G|.$
Note that
the Steinberg character vanishes for  non-semisimple conjugacy
classes and the character values of $1_G$ and $St$
are equal on split semisimple conjugacy classes, and equal but of
opposite sign at the elliptic semisimple conjugacy classes. Also, we recall that on the central elements, the value of the character $\chi_{St}$
is $p$.

Using the above facts, the total contribution of $1_G$ and $St$ to the sum in Eq. \eqref{lang} is given by
\begin{align*}
c_1\sum_{h(A) \leq x} \chi_1(A)&+c_{St}\sum_{h(A) \leq x}
  \chi_{St}(A)
=\frac{2|\Omega_{prim}|}{|G|}
S(\delta_{\Omega_e}, x) +\frac{|\Omega_{prim}|}{|G|}(1-p)x\\
&= \frac{|\Omega_{prim}|}{|G|}16(1-2/p+1/p^2)x^4+O(x^3 \sqrt{p}\log p) +O(xp),
\end{align*}
by Prop. \ref{thm:growth-elliptic}. 
Combining this estimate with  Prop. \ref{others}, Eq. \eqref{onedim}, 
estimates on Fourier coefficients of $\delta_{\Omega_{prim}}$ given by
Lemma \ref{sumfc}, and arguing as in the proof of Prop.
\ref{weaklqpr}, we obtain the asymptotic formula 
\begin{align*}
S(\delta_{\Omega_{prim}}, x)&= c_1\sum_{h(A) \leq x}\chi_1(A)
+c_{St}\sum_{h(A) \leq x} \chi_{St}(A)+
 \sum_{\rho}c_{\rho}\sum_{h(A) \leq x} \chi_{\rho}(A),\\
&=\frac{|\Omega_{prim}|}{|G|}16(1-2/p+1/p^2)x^4+O(x^3 \sqrt{p}\log p)+O(x^2p\log p)+O(p^{2+\ep}),
\end{align*}
from which the theorem follows.

\begin{remark}Note that Prop. \ref{thm:growth-elliptic} is an ingredient in
  the  proof of Theorem \ref{lqprthm}. Thus, both the $GL(1)$  version (i.e., the classical one)
  and the $GL(2)$-analogue of the  Polya-Vinogradov inequality have been used
  in   the proof of Theorem \ref{lqprthm}.
  
It is to be expected that for similar applications for  $GL(n)$ or more generally for a  reductive group $G$,  one will have to invoke Polya-Vinogradov type results attached to Levi components of parabolics in $G$. 
\end{remark}

\subsection{A different approach towards counting primitive elements}\label{Perel}
In \cite{PS}, Perel'muter and Shparlinski proves the following theorem.
\begin{theorem}\label{Pe-Sh}
Suppose $\theta\in \F_{p^n}$ is such that $\F_p (\theta)=F_{p^n}$. Then the number of integers $m, 0\leq m \leq x $ such that $\theta+m$ is a generator
of the cyclic group $\F_{p^n}^{\ast}$ is $$\frac{\phi(p^n -1)}{p^n -1} x+O(p^{1/2+\ep}).$$ 
\end{theorem}

We briefly describe how this is proved. 
After expanding the indicator function of the set of primitive roots in terms of character sums (see Lemma \ref{shapiro})
and collecting the main term arising from the trivial character, all we need is the following bound
for  non-trivial characters $\chi$ of $\F_q^{\ast}$:
\begin{equation}\label{PS}
\sum_{0\leq m\leq x}\chi(\theta+m)\ll \sqrt{p}\log p,
\end{equation}
 where  the implied constant  depends only on $n$ (in fact, one can take the constant to be $n$). 
This bound, in turn, follows by standard analytic methods (see \cite[Chap. 12]{IK}) from the bound on the complete exponential sum given below:
\begin{equation}
\sum_{m \in \F_p}\chi(\theta+m)e_p(ma) \ll \sqrt{p},
\end{equation}
for any $a \in \F_p^{\ast}$. 
This beautiful result due to Perel'muter and Shparlinski is an ingenious application of the Riemann Hypothesis for curves over finite fields proved by Weil.
See \cite{Katz} for a different approach for a special case of the above sum where the additive character is trivial.

We now give a different proof of Theorem \ref{lqprthm} using Theorem \ref{Pe-Sh} and Prop. \ref{thm:growth-elliptic}. Let us consider the set of $S(\Omega_{prim}, x)$ all integer matrices of height up to $x$ that reduces to primitive elements modulo $p$.  We partition this set according to the equivalence relation given by $A_1 \sim A_2$ if and only if $A_1- A_2$ is an integer multiple of the identity matrix. Now, each equivalence class is of the form $\{B+nI: n\in \Z, h(B+nI)\leq x\}$, where $B$ is some fixed elliptic element. For $x<p$, every element in an equivalence class is elliptic and every such class
 has $2[x]+1$ elements. Since the total number of elliptic elements of height up to $x$ is $8(1-2/p+1/p^2)x^4+O(x^3\sqrt{p}\log p)$ by Prop.  \ref{thm:growth-elliptic}, it follows that the number of equivalence class is 
 $$4(1-2/p+1/p^2)x^3+O(x^2\sqrt{p}\log p).$$ 
Now, by Theorem \ref{Pe-Sh}, every
 equivalence class has 
 $$\frac{\phi(p^2 -1)}{p^2 -1} (2x)+O(p^{1/2+\ep})$$
 many primitive elements. Therefore, after multiplication, we obtain a result of the same strength as Theorem 
 \ref{lqpr}, the difference being in the precise shape of the error term.

\end{document}